\documentclass{article}
\usepackage{latexsym}
\usepackage{amsmath, amsfonts, amsthm}
\usepackage{epsfig}
\usepackage{stmaryrd}
\usepackage{multicol}
\usepackage{pdfsync}
\usepackage{dsfont}
\usepackage{comment}
\usepackage{algorithmic}
\usepackage{algorithm}
\usepackage{caption}
\usepackage{psfrag}
\usepackage{subfig}
\usepackage{xcolor}
\usepackage{graphics}
\usepackage{enumitem}
\usepackage{mathtools}
\usepackage{bbm}
\usepackage{mathrsfs} 
\usepackage{color}
\usepackage{hyperref}
\usepackage[title]{appendix}
\usepackage{bm}
\usepackage{epstopdf}
\usepackage{tikz-cd}

\hypersetup{
	colorlinks = true, 
	urlcolor = blue, 
	linkcolor = red, 
	citecolor = blue 
}

\newcommand{\prob}{{\bf P}}

\newcommand{\e}{{\bf E}}

\newcommand{\bae}{\begin{equation}\begin{aligned}}
		\newcommand{\eae}{\end{aligned}\end{equation}}
\newcommand{\beq}{\begin{equation}}
	\newcommand{\eeq}{\end{equation}}

\newtheorem{theorem}{Theorem}[section]

\newtheorem{lemma}[theorem]{Lemma}
\newtheorem{proposition}[theorem]{Proposition}
\theoremstyle{definition}

\newtheorem{assumption}{A\!\!}
\theoremstyle{remark}
\newtheorem{remark}[theorem]{Remark}

\numberwithin{equation}{section}

\title{A Forward Propagation Algorithm for Online Optimization of Nonlinear Stochastic Differential Equations}

\author{Ziheng Wang\footnote{Mathematical Institute, University of Oxford, Oxford, OX2 6GG, UK (wangz1@math.ox.ac.uk)} \ and Justin Sirignano\footnote{Mathematical Institute, University of Oxford, Oxford, OX2 6GG, UK (Justin.Sirignano@maths.ox.ac.uk).}}

\begin{document}

\maketitle

\begin{abstract}
Optimizing over the stationary distribution of stochastic differential equations (SDEs) is computationally challenging. \cite{wang2022continuous} proposed a new forward propagation algorithm for the online optimization of SDEs. The algorithm solves an SDE, derived using forward differentiation, which provides a stochastic estimate for the gradient. The algorithm continuously updates the SDE model’s parameters and the gradient estimate simultaneously. This paper studies the convergence of the forward propagation algorithm for nonlinear dissipative SDEs. We leverage the ergodicity of this class of nonlinear SDEs to characterize the convergence rate of the transition semi-group and its derivatives. Then, we prove bounds on the solution of a Poisson partial differential equation (PDE) for the expected time integral of the algorithm's stochastic fluctuations around the direction of
steepest descent. We then re-write the algorithm using the PDE solution, which allows us to characterize the parameter evolution around the direction of steepest descent. Our main result is a convergence theorem for the forward propagation algorithm for nonlinear dissipative SDEs.
\end{abstract}


\section{Introduction}

\hspace{1.4em} 

Optimizing over the stationary distribution of a stochastic process is a challenging mathematical and computational problem. Consider a parameterized process $X^{\theta, x}_t \in \mathbb{R}^d$ which satisfies the stochastic differential equation (SDE):
\bae
\label{ergodic process}
dX_t^{\theta, x} &= \mu(X_t^{\theta, x}, \theta) dt + \sigma(X_t^{\theta, x}, \theta) dW_t, \\
X_0^{\theta, x} &= x,
\eae
where $\theta \in \mathbb{R}^\ell$ and $W_t$ is a $d$-dimensional standard Brownian motion. Suppose $X_t^{\theta, x}$ is ergodic (to be concretely specified later in the paper) with the stationary distribution $\pi_\theta$. Our goal is to select the parameters $\theta$ which minimize the objective function 
\beq
\label{objective function}
J(\theta) =  \sum_{n=1}^N \bigg{(} \e_{Y \sim \pi_\theta} \big{[} f_n(Y) \big{]} - \beta_n \bigg{)}^2,
\eeq
where $f_n$ are known functions and $\beta_n$ are the target quantities.

Optimizing over the stationary distribution $\pi_{\theta}$ of the parameterized process (\ref{ergodic process}) is challenging. For stochastic differential equations (SDEs), the standard approach is to solve a forward Kolmogorov partial differential equation (PDE) and its adjoint PDE at each optimization iteration. At each iteration, a gradient descent step is taken. If the SDE is high-dimensional, this method is computationally expensive or even intractable due to the curse-of-dimensionality. An alternative ad hoc optimization method is to simulate a trajectory of the SDE for a long time interval $[0,T]$ at each optimization iteration and then calculate a gradient by chain rule. The SDE must be re-simulated from scratch at each iteration and the calculated gradient is an approximation (with error) since $T$ is finite. Consequently, the method is computationally expensive due to constant re-simulation and furthermore has error. See \cite{wang2022continuous} for a detailed description of existing methods for optimizing the class of models (\ref{ergodic process}). 

In \cite{wang2022continuous}, a new online algorithm was developed to optimize over the stationary distribution of SDEs such as (\ref{ergodic process}). The online algorithm simultaneously simulates (\ref{ergodic process}) while continuously updating the parameter $\theta_t$ using a stochastic estimate for the gradient $\nabla_{\theta} J(\theta_t)$. The stochastic estimate for the gradient $\nabla_{\theta} J(\theta_t)$ is based upon a \emph{forward propagation} SDE for the gradient of $X_t^{\theta,x}$ with respect to $\theta$. \cite{wang2022continuous} rigorously proves convergence of the online forward propagation algorithm for a class of linear SDEs. Numerical experiments demonstrate that the forward propagation algorithm also converges for nonlinear SDEs. \emph{In this new paper, we rigorously prove convergence of the forward propagation algorithm for a class of nonlinear SDEs.}

For notational convenience (and without loss of generality), we will set $N = 1$
and $\beta_1 = \beta$ in \eqref{objective function}. The online forward propagation algorithm for optimizing \eqref{objective function} is:\footnote{In this paper's notation, the Jacobian matrix of a vector value function $f: x \in \mathbb{R}^n \to \mathbb{R}^m $ is an $m \times n$ matrix.}
\bae
\label{nonlinear update}
\frac{d\theta_t}{dt} &= -2\alpha_t \left(f(\bar X_t) - \beta \right) \left(\nabla f(X_t)\tilde X_t\right)^{\top}, \\
d \tilde X_t &= \left( \mu_x( X_t, \theta_t )\tilde X_t +\mu_\theta(X_t, \theta_t) \right) dt + \left( \sigma_x(X_t,\theta_t)\tilde X_t + \sigma_\theta(X_t, \theta_t) \right) dW_t, \\
dX_t &= \mu(X_t, \theta_t) dt + \sigma(X_t, \theta_t) dW_t, \\
d\bar X_t &= \mu( \bar X_t, \theta_t ) dt + \sigma(\bar X_t, \theta_t) d \bar W_t,
\eae
where $W_t$ and $\bar W_t$ are independent Brownian motions and $\alpha_t$ is the learning rate. $\mu_x = \frac{\partial \mu}{\partial x}, \mu_{\theta} = \frac{\partial \mu}{\partial \theta}, \sigma_x = \frac{\partial \sigma}{\partial x},$ and $\sigma_{\theta} = \frac{\partial \sigma}{\partial \theta}$. The learning rate must be chosen such that $\int_0^{\infty} \alpha_s ds = \infty$ and $\int_0^{\infty} \alpha_s^2 ds < \infty$. (An example is $\alpha_t = \frac{C}{1 + t}$.) $\tilde X_t$ estimates the derivative of $X_t$ with respect to $\theta_t$. The parameter $\theta_t$ is continuously updated using $\left(f(\bar X_t) - \beta \right) \nabla f(X_t) \tilde X_t$ as a stochastic estimate for $\nabla_{\theta} J(\theta_t)$. 

To better understand the algorithm \eqref{nonlinear update}, let us re-write the gradient of the objective function using the ergodicity of $X_t^{\theta}$:
\begin{eqnarray}
\label{gradient}
\nabla_\theta J(\theta) &=&  2 \left( \e_{\pi_\theta} f(Y) - \beta \right) \nabla_\theta \e_{\pi_\theta} f(Y) \notag \\
&\overset{a.s.}=& 2\left( \lim_{T \to \infty}\frac{1}{T} \int_0^T f(X^{\theta}_t) dt- \beta \right) \times \nabla_\theta \left( \lim_{T \to \infty}\frac1T \int_0^T f(X^\theta_t)dt \right).
\end{eqnarray}
Define $\tilde X_t^{\theta} = \nabla_\theta X_t^\theta$, which is the solution of the following SDE:
\begin{eqnarray}
d \tilde X_t^{\theta} &= \left( \mu_x( X_t^{\theta}, \theta ) \tilde X_t^{\theta} +\mu_\theta(X_t^{\theta}, \theta) \right) dt + \left( \sigma_x(X_t^{\theta},\theta) \tilde X_t^{\theta} + \sigma_\theta(X_t^{\theta}, \theta) \right) dW_t.
\end{eqnarray}
$\tilde X_t$ and $\tilde X_t^{\theta}$ satisfy the same equations, except $\theta$ is a fixed constant for $\tilde X_t^{\theta}$ while $\theta_t$ is updated continuously in time for $\tilde X_t$. If the derivative and the limit in \eqref{gradient} can be interchanged, the gradient can be expressed as
\begin{eqnarray}
\nabla_\theta J(\theta)= 2\left(\lim_{T \to \infty}\frac{1}{T} \int_0^T f(X^{\theta}_t) dt - \beta \right) \times \lim_{T \to \infty} \frac1T \int_0^T \nabla f(X^\theta_t) \tilde X_t^\theta dt,
\end{eqnarray}
which suggests a natural stochastic estimate for $\nabla_\theta J(\theta_t)$. Specifically, the forward propagation algorithm (\ref{nonlinear update}) uses 
\beq
\label{SG estimation}
G(\theta_t) := 2 \left(f(\bar X_t) - \beta \right) \nabla f(X_t) \tilde X_t
\eeq
as a stochastic estimate for $\nabla_{\theta} J(\theta_t)$. It is expected that $G(\theta_t)$ asymptotically converges to an unbiased estimate for the direction of steepest descent $\nabla_{\theta} J(\theta_t)$. 

For large $t$, we expect that $\e\left[ f(\bar X_t) - \beta \right] \approx \e_{\pi_{\theta_t}}\left[ f(Y) - \beta \right]$ and $\e \left[\nabla f(X_t) \tilde X_t \right] \approx \nabla_\theta \left( \e_{\pi_{\theta_t}}\left[ f(X) - \beta \right] \right)$ since $\theta_t$ is changing very slowly as $t$ becomes large due to $\displaystyle \lim_{t \rightarrow \infty} \alpha_t = 0$. Furthermore, since $\bar X_t$ and $X_t$ are driven by independent Brownian motions, we expect that $\e \left[ 2 \left(f(\bar X_t) - \beta \right) \nabla f(X_t) \tilde X_t \right ] \approx \nabla_{\theta} J(\theta_t)$ for large $t$ due to $\bar X_t$ and $(X_t, \tilde X_t)$ being close to conditionally independent since $\theta_t$ will be changing very slowly for large $t$. Thus, we expect that the stochastic sample $G(\theta_t) = 2 \left(f(\bar X_t) - \beta \right) \nabla f(X_t) \tilde X_t$ will provide an asymptotically unbiased estimate for the direction of steepest descent $\nabla_{\theta} J(\theta_t)$ and $\theta_t$ will converge to a local minimum of the objective function $J(\theta)$. 

The evolution of the parameters $\theta_t$ in \eqref{nonlinear update} can be analyzed by decomposing the dynamics into a gradient descent term and fluctuation terms: 
\begin{eqnarray}
\label{gradient with error}
\frac{d\theta_t}{dt} &=&  -2\alpha_t ( f(\bar X_t) - \beta ) \left( \nabla f(X_t) \tilde X_t\right)^\top \notag \\
&=& -2\alpha_t (\e_{\pi_{\theta_t}}f(Y) - \beta) \left(\nabla f(X_t) \tilde X_t\right)^\top - 2\alpha_t \left( f(\bar X_t) - \e_{\pi_{\theta_t}}f(Y) \right)  \left(\nabla f(X_t) \tilde X_t\right)^\top \notag \\
&=&  \underbrace{-\alpha_t \nabla_\theta J(\theta_t)}_{\textrm{Direction of Steepest Descent}} - \underbrace{ 2\alpha_t (\e_{\pi_{\theta_t}}f(Y)-\beta) \left( \nabla f(X_t) \tilde X_t - \nabla_\theta \e_{Y \sim \pi_{\theta_t}}f(Y) \right)^\top }_{\textrm{Fluctuation term $1$}} \notag \\
&-& \underbrace{2\alpha_t \left( f(\bar X_t) - \e_{\pi_{\theta_t}}f(Y) \right)  \left(\nabla f(X_t)\tilde X_t\right)^\top }_{\textrm{Fluctuation term $2$}}.
\end{eqnarray}
We remark here that the transpose in \eqref{gradient with error} is due to $\nabla f$ being a row vector. In \cite{wang2022continuous}, convergence of the algorithm \eqref{nonlinear update} was proven for linear SDEs. 

The main goal of this paper is to rigorously prove the convergence of algorithm \eqref{nonlinear update} for a class of nonlinear SDEs \eqref{ergodic process}. The mathematical approach uses a Poisson partial differential equation (PDE), such as in \cite{pardoux2003poisson, pardoux2001poisson}, to rewrite the fluctuation terms in terms of the solution of the PDEs. The fluctuation terms can be appropriately bounded by proving bounds on the solution to the PDEs. We leverage recent methods from \cite{rockner2021strong} to characterize the convergence rate of the transition semi-group for \eqref{ergodic process} and its derivatives with respect to the initial condition $x$ and the parameter $\theta$, which combined with the moment stability for \eqref{ergodic process}, allow us to prove there exists appropriately bounded solutions to the PDE. Once the fluctuation terms have been bounded, using the moment stability for the coupled system \eqref{nonlinear update}, we can prove convergence of the forward propagation algorithm using the cycle of stopping times argument \cite{bertsekas2000gradient, sirignano2017stochastic, sirignano2021online}.

\subsection{Contributions of this Paper}

\hspace{1.4em} We rigorously prove the convergence of the algorithm (\ref{nonlinear update}) when $\mu, \sigma$ satisfy the standard dissipative condition and their derivatives are uniformly bounded. Unlike in the traditional stochastic gradient descent algorithm, \eqref{nonlinear update} is a fully coupled system and thus the data is not i.i.d. (i.e., $X_t$ is correlated with $X_s$ for $s \neq t$) and, for a finite time $t$, the stochastic update direction $G(\theta_t)$ is not an unbiased estimate of $\nabla_{\theta} J(\theta_t)$. Thus one needs to carefully analyze the fluctuations of the stochastic update direction $G(\theta_t)$ around $\nabla_{\theta} J(\theta_t)$ and prove the stochastic fluctuations vanish in an appropriate way as $t \rightarrow \infty$.

Bounds on the fluctuations are challenging to obtain due to the online nature of the algorithm. The stationary distribution $\pi_{\theta_t}$ will continuously change as the parameters $\theta_t$ evolve. Unlike in \cite{wang2022continuous}, which studies linear SDEs whose probability density can be characterized via a closed-form formula, in this paper we study nonlinear SDEs and thus more complex calculations are required. The dissipativity of the drift and diffusion terms in \eqref{ergodic process} and the uniform bounededness for their derivatives lead to an exponential convergence rate for the transition semi-group of \eqref{ergodic process} and its derivatives with respect to the initial point $x$ and the parameter $\theta$. A Poisson PDE is constructed for the infinitesimal generator of a certain SDE system: the original nonlinear SDE and an SDE for its derivative. We prove there exists a solution to this Poisson PDE and, furthermore, the solution satisfies a polynomial bound. Then, we are able to analyze the parameter fluctuations in the online forward algorithm using the solution to the Poisson PDE. The fluctuations are re-written in terms of the solution to the Poisson PDEs using Ito's formula, the bounds for the PDE solutions are subsequently applied, and finally we can show asymptotically that the fluctuations appropriately vanish. Our main theorem proves for nonlinear dissipative SDEs that:
\begin{eqnarray}
\lim_{t \rightarrow \infty} \left| \nabla_{\theta} J(\theta_t) \right| \overset{a.s.} = 0.
\end{eqnarray}

\subsection{Literature Review}

\hspace{1.4em} Recent articles such as \cite{jin2021continuous, sharrock2020two,  sirignano2017stochastic, sirignano2020stochastic, surace2018online, wang2021global, wang2022continuous} have studied continuous-time stochastic gradient descent. Our paper has several important differences as compared to \cite{jin2021continuous, sharrock2020two,  sirignano2017stochastic, sirignano2020stochastic, surace2018online}. These previous papers estimate the parameter $\theta$ for the SDE $X_t^{\theta}$ from observations of $X_t^{\theta^{\ast}}$ where $\theta^{\ast}$ is the true parameter. In this paper, our goal is to select $\theta$ such that the stationary distribution of $X_t^{\theta}$ matches certain target statistics. Therefore, unlike the previous papers, we are directly optimizing over the stationary distribution of $X_t^{\theta}$. Examples of this optimization problem can be found in ergodic stochastic control, bayesian statistics, and reinforcement learning    \cite{carmona2021convergence, casella2001empirical,  sutton2018reinforcement}.

The analysis in this paper is also related to the literature on multi-scale models and 
their averaging principle, which arises naturally in many applications in material sciences, chemistry, fluid dynamics, biology, ecology, and climate dynamics (see \cite{pavliotis2008multiscale, weinan2003multiscale, weinan2005analysis}). There exist two components in multi-scale models, where one evolves much faster that the other. Existing averaging results for multi-scale SDEs can be found in \cite{bezemek2020large, bezemek2022moderate, bezemek2022rate, cerrai2009averaging, li2022poisson, rockner2021strong, rockner2021diffusion}. 
In these articles, Poisson equation techniques play an important role in analyzing the fluctuations of the fast SDE around the more slowly changing SDE, which has some similarity to the fluctuations of our stochastic gradient estimate $G(\theta_t)$ in \eqref{SG estimation} around the true gradient $\nabla J(\theta_t)$. However, in this paper we study a completely new multi-scale system for a novel application: an online, stochastic algorithm for optimizing over the stationary distribution of an SDE.

The presence of the $X$ process in \eqref{nonlinear update} makes the mathematical analysis challenging as the $X$ term introduces correlation across times, and this correlation does not disappear as time tends to infinity. In order for the algorithm to converge, the fluctuation terms in \eqref{gradient with error} need to decay sufficiently rapidly; we will prove this using the exponential ergodicity of the transition semi-group of $X_t^{\theta, x}$ and its derivatives with respect to the initial $x$ and the parameter $\theta$. \cite{rockner2021strong, rockner2021diffusion} use the Poisson equation technique to study the strong convergence rate of the slow component in the multi-scale models to its averaged equation. However, the theoretical results from \cite{pardoux2003poisson, pardoux2001poisson, rockner2021strong, rockner2021diffusion, sirignano2017stochastic, sirignano2020stochastic} do not apply to the multi-scale SDE system nor the corresponding Poisson PDE considered in this paper since the diffusion term in our PDE is not uniformly elliptic. This is a direct result of the process $\tilde X_t$ in (\ref{nonlinear update}), which shares the same Brownian motion with the process $X_t$. Consequently, we must analyze a new class of Poisson PDEs as in \cite{wang2022continuous}. Under the dissipative condition for nonlinear SDEs, we prove there exists a solution to this new class of Poisson PDEs which satisfies a polynomial bound. The bound on the solution is crucial for analyzing the fluctuations of the parameter evolution in the algorithm (\ref{nonlinear update}). 

\subsection{Organization of Paper}

\hspace{1.4em} The paper is organized into three main sections.  In Section \ref{main result}, we present the assumptions and the main theorem. Section \ref{main proof} rigorously proves the convergence of the online forward propagation algorithm for nonlinear dissipative SDEs.

\section{Main Results}\label{main result}

\hspace{1.4em} 
We will study the convergence of the algorithm \eqref{nonlinear update} for a class of nonlinear SDEs satisfying the following conditions. 

\begin{assumption}\label{dissipative} 
(Condition on $\mu$ and $\sigma$) There exist constants $C, \beta>0$ such that the following conditions hold for all $x_1, x_2 \in \mathbb{R}^d,\ \theta_1, \theta_2 \in \mathbb{R}^\ell$:
\begin{itemize}
\item Lipschitz continuity:
\beq
\label{Lip}
\left| \mu(x_1, \theta_1) - \mu(x_2, \theta_2) \right| + \left| \sigma(x_1, \theta_1) - \sigma(x_2, \theta_2) \right| \le C \left( |x_1-x_2| + |\theta_1 - \theta_2| \right).
\eeq
\item Dissipativity:
\beq
\label{dis condition}
\langle \mu(x_1, \theta) - \mu(x_2, \theta),\ x_1 - x_2 \rangle + \frac72 \left| \sigma(x_1, \theta) - \sigma(x_2, \theta) \right|^2 \le -\beta|x_1-x_2|^2,
\eeq
where $\langle a, \ b \rangle := b^\top a$.
\item Uniformly bounded with respect to $\theta$ at $x = 0$:
\beq
\label{0 bound}
\sup_{\theta \in \mathbb{R}^\ell} \max\{ \left| \mu(0,\theta)\right|, \left|\sigma(0,\theta)\right| \} \le C. 
\eeq
\end{itemize}
\end{assumption}

\begin{assumption}\label{first derivative}
(Conditions on first-order partial derivatives) The first-order partial derivatives $\nabla_x \mu(x, \theta), \nabla_\theta \mu(x, \theta),$ $\nabla_x \sigma(x, \theta),$ and $\nabla_\theta \sigma(x, \theta)$ exist for any $(x, \theta) \in \mathbb{R}^d\times \mathbb{R}^\ell$. For any $x_1, x_2 \in \mathbb{R}^d$,
\begin{eqnarray}
\sup_{\theta \in \mathbb{R}^\ell} \left|\nabla_x \mu(x_1, \theta) - \nabla_x \mu(x_2, \theta) \right| \le C|x_1-x_2|, \\
\sup_{\theta \in \mathbb{R}^\ell} \left|\nabla_\theta \mu(x_1, \theta) - \nabla_\theta \mu(x_2, \theta) \right| \le C|x_1-x_2|, \\
\sup_{\theta \in \mathbb{R}^\ell} \left|\nabla_x \sigma(x_1, \theta) - \nabla_x \sigma(x_2, \theta) \right| \le C|x_1-x_2|, \\
\sup_{\theta \in \mathbb{R}^\ell} \left|\nabla_\theta \sigma(x_1, \theta) - \nabla_\theta \sigma(x_2, \theta) \right| \le C|x_1-x_2|.
\end{eqnarray}	
\end{assumption}

\begin{assumption}\label{high derivative}
(Conditions on higher-order partial derivatives) The second-order partial derivatives $\nabla^2_{x} \mu(x, \theta), \nabla^2_{\theta} \mu(x, \theta)$, $\nabla_x\nabla_{\theta} \mu(x, \theta),$ and  $\nabla_x^2 \nabla_{\theta} \mu(x, \theta)$ exist for any $(x, \theta) \in \mathbb{R}^d \times \mathbb{R}^\ell$. For any $x_1, x_2 \in \mathbb{R}^d$,
\begin{eqnarray}
\sup_{\theta \in \mathbb{R}^\ell} \left|\nabla^2_{x} \mu(x_1, \theta) - \nabla^2_{x} \mu(x_2, \theta) \right| \le C|x_1-x_2|, \label{high 1}\\
\sup_{\theta \in \mathbb{R}^\ell} \left|\partial^2_{\theta} \mu(x_1, \theta) - \nabla^2_{\theta} \mu(x_2, \theta) \right| \le C|x_1-x_2|, \label{high 2}\\
\sup_{\theta \in \mathbb{R}^\ell} \left|\nabla_{x}\nabla_\theta \mu(x_1, \theta) - \nabla_{x}\nabla_\theta \mu(x_2, \theta) \right| \le C|x_1-x_2|, \label{high 3}\\
\sup_{\theta \in \mathbb{R}^\ell} \left|\nabla_x^2 \nabla_{\theta} \mu(x_1, \theta) - \nabla_x^2 \nabla_{\theta} \mu(x_2, \theta) \right| \le C|x_1-x_2|. \label{high 4}
\end{eqnarray}
Furthermore, if $\mu$ is replaced by $\sigma$, the properties \eqref{high 1} - \eqref{high 4} also hold, and 
\bae
\sup_{(x, \theta) \in \mathbb{R}^d \times \mathbb{R}^\ell} \max\left\{ \left| \nabla_x^2 \mu(x, \theta) \right|,\ \left| \nabla^2_{\theta} \mu(x, \theta) \right|,\ \left|\nabla_x \nabla_{\theta} \mu(x, \theta) \right|,\ \left|\nabla_x^2 \nabla_{\theta} \mu(x, \theta) \right| \right\} &\le C, \\
\sup_{(x, \theta) \in \mathbb{R}^d \times \mathbb{R}^\ell} \max\left\{ \left| \nabla_x^2 \sigma(x, \theta) \right|,\ \left| \nabla^2_{\theta} \sigma(x, \theta) \right|,\ \left|\nabla_x \nabla_{\theta} \sigma(x, \theta) \right|,\ \left| \nabla_x^2 \nabla_{\theta} \sigma(x, \theta) \right| \right\} &\le C,
\eae
\end{assumption}

\begin{assumption}\label{f}
	The function $f$ in the objective function is continuously differentiable and has  uniformly bounded derivatives, i.e. there exists a constant $C$ such that
	\beq
	\label{gradient f bound}
	\left| \nabla^i f(x) \right| \le C, \quad \forall x\in \mathbb{R}^d,\  i= 1,2,3.
	\eeq
\end{assumption}

\begin{assumption}\label{lr}
	The learning rate $\alpha_t$ satisfies $\int_{0}^{\infty} \alpha_{t} d t=\infty$, $\int_{0}^{\infty} \alpha_{t}^{2} d t<\infty$, $\int_{0}^{\infty}\left|\alpha_{s}^{\prime}\right| d s<\infty$, and there is a $p>0$ such that $\displaystyle \lim _{t \rightarrow \infty} \alpha_{t}^{2} t^{\frac12 + 2 p}=0$.
\end{assumption}

\begin{remark} 
Our assumptions are standard in the mathematical literature which studies ergodic SDEs. We briefly comment on these assumptions before presenting our main theoretical result.
\begin{itemize}
    \item A sufficient condition for the dissipative SDE \eqref{ergodic process} to satisfy Assumptions A\ref{dissipative}-A\ref{high derivative} is that the first-, second-, and third-order derivatives of $\mu, \sigma$ are uniformly bounded and for any $x, y \in \mathbb{R}^d, \theta \in \mathbb{R}^\ell$
    \bae
    \label{dissipative 2}
    y^\top \nabla_x \mu(x, \theta) y \le -C|y|^2,\quad |\sigma(x,\theta)-\sigma(y,\theta)| \le L|x-y|,
    \eae
    where $C, L>0$ are constants and $\frac32 L^2<C$. A classic example is the Langevin Equation, where the drift term is the gradient of some convex potential. That is $\mu(x, \theta) = -\nabla V(x, \theta)$ with $V(x, \theta)$ being convex with respect to $x$. See  \cite{pavliotis2014stochastic, pavliotis2007parameter} for a detailed discussion.
    \item  Conditions \eqref{Lip} and \eqref{0 bound} imply that there exists a constant $C>0$ such that, for any $x \in \mathbb{R}^{d}$ and $\theta \in \mathbb{R}^{\ell}$, 
    \beq
    |\mu(x, \theta)| + |\sigma(x, \theta)| \le C \left(1+|x|\right).
    \eeq
    \item Conditions \eqref{dis condition} and \eqref{0 bound} imply that there exists a constant $C>0$ such that, for any $x \in \mathbb{R}^{d}$ and $\theta \in \mathbb{R}^{\ell}$, 
    \beq
    \label{4 moment dis}
    2 \langle \mu(x, \theta),\ x\rangle + 7|\sigma(x, \theta)|^{2} \le -\beta |x|^{2} + C.
    \eeq
    The derivation of the above inequality is:
    \bae
    \label{trick}
    &2 \langle \mu(x, \theta),\ x\rangle + 7|\sigma(x, \theta)|^{2} \\
    =& 2 \langle \mu(x, \theta) - \mu(0, \theta),\ x\rangle + 2\langle \mu(0, \theta),\ x\rangle + 7|\sigma(x, \theta)-\sigma(0, \theta) + \sigma(0, \theta)|^{2} \\
    \overset{(a)}{\le}& -2\beta  |x|^2 + 2\langle \mu(0, \theta),\ x\rangle + 7 |\sigma(0, \theta)|^2 + 14|\sigma(x,\theta)- \sigma(0, \theta)| \cdot | \sigma(0,\theta)|\\
    \overset{(b)}{\le}& -\beta|x|^2 + C \left( |\mu(0,\theta)|^2 + |\sigma(0,\theta)|^2 \right),
    \eae
    where step $(a)$ uses the inequality \eqref{dis condition} and step $(b)$ use Young's inequality and the inequality \eqref{Lip}.
    
    Condition \eqref{dis condition} is used to prove the solution of dynamic \eqref{nonlinear update} has uniformly bounded fourth moment.
    \item Assumption (A\ref{dissipative}) guarantees (see Theorem 4.3.9 of \cite{prevot2007concise}) that there exists a unique invariant measure $\pi_\theta$ for \eqref{ergodic process} such that  
    \beq
    \label{stationary moment}
    \int_{\mathbb{R}^d} |x| \pi_\theta(dx) \le C<\infty.
    \eeq
    \end{itemize}
\end{remark}

Under these assumptions, we are able to prove the convergence of the online forward algorithm \eqref{nonlinear update}.

\begin{theorem}
	\label{conv}
	Under Assumptions (A\ref{dissipative}) - (A\ref{lr}) and for the SDE system \eqref{nonlinear update}, we have
	\beq
	\lim_{t \rightarrow \infty} \left| \nabla_\theta J(\theta_t) \right|  \overset{a.s.} =  0. 
	\eeq
\end{theorem}

\section{Proof}\label{main proof}

\hspace{1.4em} The SDE system \eqref{nonlinear update} has a unique strong solution.\footnote{Existence and uniqueness can be proven using the standard method of a contraction map; see Theorem 1.2 of \cite{carmona2016lectures} for details.} In equation \eqref{gradient with error}, we decomposed the evolution of $\theta_t$ into the direction of steepest descent $-\alpha_t \nabla_\theta J(\theta_t)$ and two fluctuation terms. Define the fluctuation terms as 
\bae
\label{error}
Z_t^1 &= (\e_{\pi_{\theta_t}}f(Y)-\beta) \left( \nabla f(X_t) \tilde X_t -  \nabla_\theta \e_{\pi_{\theta_t}}f(Y) \right)^\top, \\
Z_t^2 &= \left( f(\bar X_t) - \e_{\pi_{\theta_t}}f(Y)) \right) \left( \nabla f(X_t) \tilde X_t\right)^\top.
\eae
As in \cite{sirignano2017stochastic}, we will study a cycle of stopping times to control the time periods where $|\nabla_\theta J(\theta_t)|$ is close to zero and away from zero. Let us select an arbitrary constant $\kappa>0$ and also define $\mu=\mu(\kappa)>0$ (to be chosen later). Then, set $\sigma_{0}=0$ and define the cycles of random times
$$
0=\sigma_{0} \leq \tau_{1} \leq \sigma_{1} \leq \tau_{2} \leq \sigma_{2} \leq \ldots,
$$
where the stopping times are defined as
\bae
	\label{cycle of time}
	&\tau_{n}=\inf \left\{t>\sigma_{n-1}:\left|\nabla_\theta  J\left(\theta_{t}\right)\right| \geq \kappa\right\} \\
	&\sigma_{n}=\sup \left\{t>\tau_{n}: \frac{\left|\nabla_\theta J\left(\theta_{\tau_{n}}\right)\right|}{2} \leq\left|\nabla _\theta J\left(\theta_{s}\right)\right| \leq 2\left|\nabla_\theta J\left(\theta_{\tau_{n}}\right)\right| \text { for all } s \in\left[\tau_{n}, t\right] \text { and } \int_{\tau_{n}}^{t} \alpha_{s} d s \leq \mu \right\}.
\eae 
We define the random time intervals $J_{n}=\left[\sigma_{n-1}, \tau_{n}\right)$ and $I_{n}=\left[\tau_{n}, \sigma_{n}\right)$. We introduce the constant $\eta > 0$ which will be chosen to be sufficiently small later. In order to prove convergence, we will have to show that the fluctuation terms become small as $t \rightarrow \infty$. In particular, the following integral of the fluctuation term will be crucial to the convergence analysis:
\beq
\label{error integral}
\Delta^i_{\tau_n,\sigma_n + \eta} := \int_{\tau_n}^{\sigma_n + \eta} \alpha_s Z^i_s ds, \quad i =1,2.
\eeq

We will begin our analysis by first presenting several lemmas regarding Lipschitz continuity, moment bounds, and ergodicity. The proofs are the same as in \cite{rockner2021strong} and thus we omit them.

\begin{lemma}[Lipschitz continuity]
\label{difference lemma}
For any $t>0$, $x_i\in \mathbb{R}^d$, and $\theta_i \in R^{\ell}$, we have
\beq
\e\left| X_t^{\theta_1, x_1} - X_t^{\theta_2, x_2}  \right|^2 \le e^{-\beta t} |x_1 - x_2|^2 + C|\theta_1 - \theta_2|^2.
\eeq
\end{lemma}
A proof can be found in Lemma $3.6$ of \cite{rockner2021strong}.

\begin{lemma}[Ergodicity]
\label{ergodic}
For any $t\ge 0$, $x \in \mathbb{R}^d$, and $\theta \in \mathbb{R}^\ell$,
\beq
\left| \e f(X_t^{\theta,x}) - \e_{\pi_\theta} f(Y) \right| \le C e^{-\frac{\beta t}{2}} (1+|x|).
\eeq
\end{lemma}
A proof can be found in Proposition $3.7$ of \cite{rockner2021strong}.

\begin{lemma}[Moment Bound]
\label{moment stable} 
There exists a constant $C$ such that
\beq
\e \left|X_t^{\theta, x}\right|^2 \le C(1+e^{-\beta t}|x|^2), \quad \forall x\in\mathbb{R}^d, t\ge 0.
\eeq
\end{lemma}

\begin{proof}
Using It\^{o}'s formula to $\left| X_t^{\theta, x}\right|^2$, we have
\bae
\frac{d}{dt} \e \left| X_t^{\theta, x} \right|^2 &= \e \left[ 2\left\langle \mu(X_t^{\theta, x}, \theta),\ X_t^{\theta, x} \right\rangle + \left| \sigma(X_t^{\theta, x}, \theta) \right|^2 \right] \\
&\le \e \left[ 2\left\langle \mu(X_t^{\theta, x}, \theta) -\mu(0, \theta),\ X_t^{\theta, x} \right\rangle + 2\left| \sigma(X_t^{\theta, x}, \theta) - \sigma(0, \theta) \right|^2 + 2\left\langle \mu(0, \theta),\ X_t^{\theta, x} \right\rangle + 2\left| \sigma(0,\theta) \right|^2 \right]\\
&\overset{(a)}{\le} -2\beta \e \left| X_t^{\theta, x} \right|^2 + \left( \beta \e \left|X_t^{\theta, x}\right|^2 + \frac1\beta \left| \mu(0, \theta) \right|^2 \right) + 2\left| \sigma(0,\theta) \right|^2 \\
&\overset{(a)}{\le} -\beta \e \left| X_t^{\theta, x} \right|^2 + C,
\eae
where step $(a)$ uses the dissipativity assumption \eqref{dis condition} and Young's inequality and step $(b)$ uses the bound \eqref{0 bound}. Therefore, using a comparison principle for ODEs,
\beq
\label{moment bound}
\e \left| X_t^{\theta, x} \right|^2 \le e^{-\beta t} |x|^2 + C.
\eeq
\end{proof}

Using similar calculations as in Proposition 4.1 of \cite{rockner2021strong}, several ergodicity results for $X^\theta_t$ can be proven.
\begin{proposition}
    \label{ergodic estimation}
    Under Assumptions (A\ref{dissipative}) - (A\ref{f}), we have the following ergodic bounds:
	\begin{itemize}
	    \item[(\romannumeral1)] There exists a constant $C$ such that for any $\theta\in \mathbb{R}^\ell, x \in \mathbb{R}^d$, and $t>0$, 
		\beq
		\label{theta decay}
		\left| \nabla_\theta^i \e f(X_t^{\theta,x}) - \nabla_\theta^i \e_{\pi_\theta}f(Y) \right| \leq C e^{-\beta t} (1+|x|), \quad i=0,1,2.
		\eeq
		\item[(\romannumeral2)] There exists a constant $C>0$ such that for any $\theta\in \mathbb{R}^\ell$ and $i = 0, 1, 2$,
		\beq
		\label{invariant density}
		\left|\nabla_\theta^i \e_{\pi_\theta} f\left(Y\right)\right| \leq C.
		\eeq
		\item[(\romannumeral3)] There exists constants $C, \gamma>0$ such that for any for any $\theta\in \mathbb{R}^\ell, x \in \mathbb{R}^d$, and $t>0$,
		\beq
		\label{theta x decay}
		\left|\nabla_{x}^{j} \nabla_\theta^i \e f(X_t^{\theta,x})\right| \leq Ce^{-\gamma t}, \quad i = 0, 1, \quad j = 1,2.
		\eeq
	\end{itemize}
\end{proposition} 

The proof method for Proposition \ref{ergodic estimation} is the same as in Proposition $4.1$ of \cite{rockner2021strong}, although we need the convergence result for higher-order derivatives in \eqref{theta x decay}. For completeness, we provide the detailed proof for all orders of the derivatives in Appendix \ref{detailed ergodic}.

We next prove that a solution exists to a Poisson equation for the fluctuation terms and, furthermore, that the solution satisfies certain polynomial bounds. We first introduce the process $\tilde X_t^{\theta, x, \tilde x}$, which satisfies the SDE:
\begin{equation}
    \label{tilde sde}
	\left\{
	\begin{aligned}
		d \tilde X_t^{\theta, x, \tilde x} &= \left[ \nabla_{x} \mu(X_t^{\theta, x}, \theta) \tilde X_t^{\theta, x, \tilde x} + \nabla_\theta \mu(X_t^{\theta, x}, \theta)\right] dt + \left[ \nabla_{x} \sigma(X_t^{\theta, x}, \theta) \tilde X_t^{\theta, x, \tilde x} + \nabla_\theta \sigma(X_t^{\theta, x}, \theta) \right]dW_t, \\
        \tilde X_0^{\theta, x, \tilde x}&= \tilde x,
	\end{aligned}
	\right.
\end{equation}
where the Brownian is the same as in \eqref{ergodic process}. It should be noted that $\tilde X_t^{\theta, x, 0} = \nabla_\theta X_t^{\theta, x}$ almost surely.

\begin{lemma}
\label{poisson eq}
Define the error function 
\beq
\label{function}
G^1(x,\tilde x, \theta) = (\e_{\pi_\theta} f(Y)-\beta) \left(\nabla f(x) \tilde x - \nabla_\theta \e_{\pi_\theta}f(Y) \right)^\top
\eeq
and the function
\beq
\label{representation}
v^1(x, \tilde x, \theta) = -\int_0^\infty \e G^1(X_t^{\theta, x},\tilde X_t^{\theta, x, \tilde x}, \theta) dt.
\eeq
Let $\mathcal{L}^\theta_{x,\tilde x}$ denote the infinitesimal generator of the process $(X_\cdot^{\theta, x}, \tilde X_\cdot^{\theta,x,\tilde x})$, i.e. for any test function $\varphi$
$$
\begin{aligned}
\mathcal{L}^{\theta}_{x,\tilde x}\varphi\left(x, \tilde x \right) &= \mathcal{L}^{\theta}_{x} \varphi(x, \tilde x) + \sum_{k=1}^\ell \mathcal{L}^\theta_{\tilde x^{:, k}} \varphi(x, \tilde x) \\
&+ \sum_{j=1}^\ell \text{tr}\left(\nabla_{\tilde x^{:,j}}\nabla_x \varphi(x, \tilde x) \sigma(x,\theta)\left(\nabla_x \sigma(x,\theta)\tilde x^{:,j} + \frac{\partial \sigma(x, \theta)}{\partial \theta_j}\right)^\top\right)\\
&+ \sum_{j<k} \text{tr}\left(\nabla_{\tilde x^{:,k}}\nabla_{\tilde x^{:,j}} \varphi(x, \tilde x) \left(\nabla_x \sigma(x,\theta)\tilde x^{:,j} + \frac{\partial \sigma(x, \theta)}{\partial \theta_j}\right) \left(\nabla_x \sigma(x,\theta)\tilde x^{:,k} + \frac{\partial \sigma(x, \theta)}{\partial \theta_k}\right)^\top\right)
\end{aligned}
$$
where $\tilde x^{:, k}$ for $k \in \{1,\cdots, \ell \}$ is the k-th column of $\tilde x$.

Then, under Assumptions (A\ref{dissipative}) - (A\ref{f}), $v^1(x,\tilde x, \theta)$ is the classical solution of the Poisson equation 
\beq
\label{PDE}
\mathcal{L}_{x,\tilde x}^\theta u(x,\tilde x, \theta) = G^1(x,\tilde x, \theta),
\eeq
where $u = (u_1, \ldots, u_\ell)^\top \in \mathbb{R}^\ell$ is a vector,  $\mathcal{L}_{x,\tilde x}^\theta u(x,\tilde x, \theta) = (\mathcal{L}_{x,\tilde x}^\theta u_1(x,\tilde x, \theta), \ldots, \mathcal{L}_{x,\tilde x}^\theta u_{\ell}(x,\tilde x, \theta))^\top$.
Furthermore, the solution $v^1$ satisfies the bound
\bae
\label{control v1}
\left| v^1(x, \tilde x, \theta)\right| + \left|\nabla_\theta v^1(x, \tilde x, \theta)\right|+\left|\nabla_x v^1(x, \tilde x, \theta)\right| + \left|\nabla_{\tilde x} v^1(x, \tilde x, \theta)\right| \le C\left(1+ |x| + |\tilde x| \right),
\eae
where $C > 0$ is a constant which does not depend upon $(x,\tilde x, \theta)$.
\end{lemma}

\begin{proof}
We begin by proving that the integral \eqref{representation} is finite. We divide \eqref{representation} into two terms: 
\bae
v^1(x,\tilde x, \theta) &= ( \e_{\pi_\theta} f(Y) - \beta) \int_0^\infty \left( \nabla_\theta \e_{\pi_\theta}f(Y) - \e \nabla f(X_t^{\theta, x}) \tilde X_t^{\theta, x, \tilde x} \right)^\top dt \\
&= (\e_{\pi_\theta} f(Y) - \beta) \left[ \int_0^\infty \left( \nabla_\theta \e_{\pi_\theta}f(Y) - \nabla_\theta \e f(X_t^{\theta, x}) \right)^\top dt + \int_0^\infty \left( \nabla_\theta \e f(X_t^{\theta, x}) - \e \nabla f(X_t^{\theta, x}) \tilde X_t^{\theta, x, \tilde x} \right)^\top dt \right]\\
&=: v^{1,1}(x, \theta) + v^{1,2}(x, \tilde x, \theta).
\eae

We first bound $v^{1,1}(x, \theta)$. Following the method in Lemma $3.3$ of  \cite{wang2022continuous}, we have by Proposition \ref{ergodic estimation} and the dominated convergence theorem (DCT) that:
\bae
\left| v^{1,1}(x, \theta)\right| & \le C \int_0^\infty \left| \nabla_\theta \e_{\pi_\theta}f(Y) - \nabla_\theta \e f(X_t^{\theta, x}) \right| dt \le C(1+|x|),\\
\left|\nabla_\theta v^{1,1}(x, \theta)\right| & \le C \int_0^\infty \left| \nabla_\theta \e_{\pi_\theta}f(Y) - \nabla_\theta \e f(X_t^{\theta, x}) \right| +  C \int_0^\infty \left| \nabla^2_\theta \e_{\pi_\theta}f(Y) - \nabla^2_\theta \e f(X_t^{\theta, x}) \right| \le C(1+|x|),
\\
\left|\nabla^i_x v^{1,1}(x, \theta)\right| &\le C \int_0^\infty \left| \nabla_x^i \nabla_\theta \e f(X_t^{\theta, x}) \right| dt \le  C, \quad i =1, 2.
\eae

For $ v^{1,2}(x, \tilde x, \theta)$, 
define 
$$
Z_t = \tilde  X_t^{\theta, x. \tilde x_1} - \tilde X_t^{\theta, x, \tilde x_2}.
$$ 
We can derive a differential inequality for $Z^{:,k}_t$, the k-th column of $Z_t$, using the inequality (\ref{MuDerivInequality}):
\beq
\label{ito difference}
\frac{d}{dt} \e \left| Z^{:,k}_t \right|^2 \overset{(a)}{=} \e \left[ 2\left\langle \nabla_{x} \mu(X_t^{\theta, x_1}, \theta) Z^{:,k}_t,\ Z^{:,k}_t \right\rangle + \left| \nabla_{x} \sigma(X_t^{\theta, x_1}, \theta)Z^{:,k}_t \right|^2 \right] \le -\beta \e|Z^{:,k}_t|^2,
\eeq
where step $(a)$ is by using It\^{o}'s formula to $\left|Z^{:,k}_t \right|^2$. Therefore, we can prove the exponential decay:
\begin{eqnarray}
\label{tilde lip}
\e \left|\tilde X_t^{\theta, x, \tilde x_1} - \tilde X_t^{\theta, x, \tilde x_2} \right|^2 &\leq& C e^{-\beta t} |\tilde x_1 - \tilde x_2|^2, \notag \\
\e \left| \nabla_{\tilde x} \tilde X_t^{\theta, x, \tilde x} \right|^2 &\leq& C e^{-\beta t}.
\end{eqnarray}
Let $\tilde X^{\theta, x, \tilde x, :, k}$ denote the k-th column of the matrix $\tilde X_t^{\theta, x, \tilde x}$ and for any $m \in \{1,\cdots, d\}, n \in \{1, \cdots, \ell\}$, we know
\beq
d \frac{\partial \tilde X_t^{\theta, x, \tilde x, :, k}}{\partial \tilde x^{m,n}} = \nabla_x \mu(X_t^{\theta, x}, \theta) \frac{\partial \tilde X_t^{\theta, x, \tilde x, :, k}}{\partial \tilde x^{m,n}} dt + \nabla_x \sigma(X_t^{\theta, x}, \theta) \frac{\partial \tilde X_t^{\theta, x, \tilde x, :, k}}{\partial \tilde x^{m,n}} dW_t,
\eeq
where $\tilde x^{m,n}$ denotes the $(m, n)$ element of the matrix $\tilde x$. Let
$$
\tilde Z^1_t = \frac{\partial \tilde X_t^{\theta, x, \tilde x_1, :, k}}{\partial \tilde x^{m,n}} - \frac{\partial \tilde X_t^{\theta, x, \tilde x_2, :, k}}{\partial \tilde x^{m,n}}, \quad \tilde Z^2_t = \frac{\partial \tilde X_t^{\theta, x_1, \tilde x, :, k}}{\partial \tilde x^{m,n}} - \frac{\partial \tilde X_t^{\theta, x_2, \tilde x, :, k}}{\partial \tilde x^{m,n}}.
$$ 
Note that $\tilde Z_t^1$ satisfies the SDE
\beq
d\tilde Z^1_t = \nabla_x \mu(X_t^{\theta, x}, \theta) \tilde Z^1_t dt + \nabla_x \sigma(X_t^{\theta, x}, \theta) \tilde Z^1_t dW_t
\eeq
Similar to \eqref{ito difference}, we can get
\beq
\frac{d}{dt} \e \left| \tilde Z^1_t \right|^2 \le -\beta \e\left| \tilde Z^1_t \right|^2 
\eeq
which derives 
\begin{eqnarray}
\label{tilde lip 2}
\e \left|\nabla_{\tilde x} \tilde X_t^{\theta, x, \tilde x_1} - \nabla_{\tilde x} \tilde X_t^{\theta, x, \tilde x_2} \right|^2 &\leq& C e^{-\beta t} |\tilde x_1 - \tilde x_2|^2, \notag \\
\e \left| \nabla^2_{\tilde x} \tilde X_t^{\theta, x, \tilde x} \right|^2 &\leq& C e^{-\beta t}.
\end{eqnarray}
Then for $\tilde Z_t^2$
\bae
 d \tilde Z_t^2 &= \left( \nabla_x \mu(X_t^{\theta, x_1}, \theta) \frac{\partial \tilde X_t^{\theta, x_1, \tilde x, :, k}}{\partial \tilde x^{m,n}} - \nabla_x \mu(X_t^{\theta, x_2}, \theta) \frac{\partial \tilde X_t^{\theta, x_2, \tilde x, :, k}}{\partial \tilde x^{m,n}} \right) dt \\
&+ \left( \nabla_x \sigma(X_t^{\theta, x_1}, \theta) \frac{\partial \tilde X_t^{\theta, x_1, \tilde x, :, k}}{\partial \tilde x^{m,n}} - \nabla_x \sigma(X_t^{\theta, x_2}, \theta) \frac{\partial \tilde X_t^{\theta, x_2, \tilde x, :, k}}{\partial \tilde x^{m,n}} \right) dW_t,
\eae
and as in \eqref{grad theta lip cal}
\bae
\frac{d}{dt} \e \left| \tilde Z_t^2 \right|^2 =& \e \left[ 2\left\langle \nabla_x \mu(X_t^{\theta, x_1}, \theta) \frac{\partial \tilde X_t^{\theta, x_1, \tilde x, :, k}}{\partial \tilde x^{m,n}} - \nabla_x \mu(X_t^{\theta, x_2}, \theta) \frac{\partial \tilde X_t^{\theta, x_2, \tilde x, :, k}}{\partial \tilde x^{m,n}},\ \tilde Z^2_t \right\rangle \right]\\
+& \e \left[ \left| \nabla_x \sigma(X_t^{\theta, x_1}, \theta) \frac{\partial \tilde X_t^{\theta, x_1, \tilde x, :, k}}{\partial \tilde x^{m,n}} - \nabla_x \sigma(X_t^{\theta, x_2}, \theta) \frac{\partial \tilde X_t^{\theta, x_2, \tilde x, :, k}}{\partial \tilde x^{m,n}} \right|^2 \right] \\
\le& \e \left[ 2\left\langle \nabla_{x} \mu(X_t^{\theta, x_1}, \theta) \tilde Z_t^2,\ \tilde Z^2_t \right\rangle + 2\left| \nabla_{x} \sigma(X_t^{\theta, x_1}, \theta) \tilde Z_t^2 \right|^2\right]
+ \beta \e |\tilde Z^2_t|^2 + C \e \left| X_t^{\theta, x_1} - X_t^{\theta, x_2}\right|^2 \\
\le& -\beta \e|\tilde Z^2_t|^2 + Ce^{-\beta t}|x_1 - x_2|^2,
\eae
which derives
\begin{eqnarray}
\label{tilde x lip}
\e \left|\nabla_{\tilde x} \tilde X_t^{\theta, x_1, \tilde x} - \nabla_{\tilde x} \tilde X_t^{\theta, x_2, \tilde x} \right|^2 &\leq& C e^{-\beta t} |x_1 - x_2|^2, \notag \\
\e \left| \nabla_x \nabla_{\tilde x} \tilde X_t^{\theta, x, \tilde x} \right|^2 &\leq& C e^{-\beta t}.
\end{eqnarray}

Combining \eqref{tilde lip}, \eqref{tilde lip 2} and \eqref{tilde x lip} we can establish bounds on $v^{1,2}(x, \tilde x, \theta)$.
\bae
\left| v^{1,2}(x, \tilde x, \theta) \right| &\overset{(a)}{\le} C \int_0^\infty \e \left| \nabla f(X_t^{\theta, x}) \left(\tilde X_t^{\theta, x, \tilde x} - \tilde X_t^{\theta, x, 0} \right) \right| dt \le \int_0^\infty Ce^{-\frac\beta2 t}|\tilde x|dt \le C |\tilde x|\\
\left| \nabla^i_{\tilde x} v^{1,2}(x, \tilde x, \theta) \right| &\le C \int_0^\infty \e \left| \nabla^i_{\tilde x} \tilde X_t^{\theta, x, \tilde x} \right| dt \le \int_0^\infty Ce^{-\frac\beta2 t}dt \le C, \quad i=1,2 \\
\left| \nabla_x \nabla_{\tilde x} v^{1,2}(x, \tilde x, \theta) \right| &\le C \int_0^\infty \e \left[ \left|\nabla_x^2 f(X_t^{\theta,x}) \nabla_x X_t^{\theta,x}\right| \cdot \left|\nabla_{\tilde x} X_t^{\theta,x, \tilde x}\right| \right] dt + C \int_0^\infty \e \left| \nabla_x \nabla_{\tilde x} \tilde X_t^{\theta, x, \tilde x} \right| dt \le C
\eae
where in step $(a)$ we use the fact $\nabla_\theta X_t^{\theta, x} \overset{a.s.}{=} \tilde X_t^{\theta, x, 0}$. 

The analysis of $\nabla^i_x v^{1,2}$ for $i =1,2$ and $\nabla_\theta v^{1,2}$ is similar to the calculations for $v^{1,1}$. Define
$$
\bar Z_t = \nabla_\theta \tilde  X_t^{\theta, x. \tilde x_1} - \nabla_\theta \tilde X_t^{\theta, x, \tilde x_2}.
 $$ 
$\bar Z_t$ satisfies the SDE: 
\bae
\label{bar Z}
d \bar Z_t &= \left( \left\langle \nabla_x^2 \mu(X_t^{\theta, x}, \theta) \nabla_\theta X_t^{\theta, x},\  Z_t \right\rangle + \nabla_x \nabla_\theta \mu(X_t^{\theta, x}, \theta) Z_t + \nabla_x \mu(X_t^{\theta, x}, \theta) \bar Z_t\right) dt, \\
&+ \left(\left\langle \nabla_x^2 \sigma(X_t^{\theta, x}, \theta) \nabla_\theta X_t^{\theta, x},\  Z_t \right\rangle + \nabla_x \nabla_\theta \sigma(X_t^{\theta, x}, \theta) Z_t + \nabla_x \sigma(X_t^{\theta, x}, \theta) \bar Z_t\right) dW_t,
\eae
where $ \left\langle \ ,\ \right\rangle$ in the equation above is defined as:
\bae
\label{bar Z element}
\left\langle \nabla_x^2 \mu(X_t^{\theta, x}, \theta) \nabla_\theta X_t^{\theta, x},\  Z_t \right\rangle_{m,p,q} &= \left\langle \nabla_x^2 \mu_m(X_t^{\theta,x}, \theta)\frac{\partial X_t^{\theta, x}}{\partial \theta_q},\ Z_t^{:, p} \right\rangle,\\
\left\langle \nabla_x^2 \sigma(X_t^{\theta, x}, \theta) \nabla_\theta X_t^{\theta, x},\  Z_t \right\rangle_{m,n,p,q} &= \left\langle \nabla_x^2 \sigma_{mn}(X_t^{\theta,x}, \theta)\frac{\partial X_t^{\theta, x}}{\partial \theta_q},\ Z_t^{:, p} \right\rangle.
\eae
Thus by the same calculations as in \eqref{grad theta lip cal} and the uniform bounds for the derivatives of $\mu, \sigma$, we can derive the differential inequality:
\beq
\frac{d}{dt} \e \left| \bar Z_t \right|^2 \le -\beta \e|\bar Z_t|^2 + C \e |Z_t|^2.
\eeq
Using an integrating factor, we have 
\beq
\frac{d}{dt} \left( e^{\beta t} \e \left| \bar Z_t \right|^2\right) \le Ce^{\beta t} \e \left| Z_t \right|^2,
\eeq
which combined with \eqref{tilde lip} yields
\begin{eqnarray}
\label{tilde theta lip}
\e |\nabla_\theta \tilde X_t^{\theta, x, \tilde x_1} - \nabla_\theta \tilde X_t^{\theta, x, \tilde x_2}|^2 &=& \e \left| \bar Z_t \right|^2 \le e^{-\beta t} \left| \tilde x_1 - \tilde x_2 \right|^2 + e^{-\beta t} \int_0^t e^{\beta s} \e \left| Z_s \right|^2 ds \notag \\
&\le& C e^{-\frac\beta2 t} |\tilde x_1 - \tilde x_2|^2.
\end{eqnarray}
Consequently,
\bae
\left| \nabla_\theta v^{1,2}(x, \tilde x, \theta) \right|
\le& C |\tilde x| \cdot \left|\nabla_\theta \e_{\pi_\theta} f(Y)\right| + C  \int_0^\infty \e \left| \nabla f(X_t^{\theta, x}) \left(\nabla_\theta \tilde X_t^{\theta, x, \tilde x} - \nabla_\theta \tilde X_t^{\theta, x, 0} \right) \right| dt \\
+& \int_0^\infty  \e \left| \left\langle \nabla^2 f(X_t^{\theta, x}) \nabla_\theta X_t^{\theta, x},\ \tilde X_t^{\theta, x, \tilde x} - \tilde X_t^{\theta, x, 0}  \right\rangle \right| dt \\
\overset{(a)}{\le}& C|\tilde x| + \int_0^\infty Ce^{-\frac\beta4 t}|\tilde x|dt + \int_0^\infty Ce^{-\frac\beta2 t}|\tilde x|dt \\
\le& C |\tilde x|,
\eae
where in step $(a)$ we use the Cauchy-Schwarz inequality, \eqref{grad theta bound}, \eqref{tilde lip} and \eqref{tilde theta lip}.

Finally, for the derivatives with respect to $x$, define 
$$
\hat Z_t = \nabla_x \tilde  X_t^{\theta, x. \tilde x_1} - \nabla_x \tilde X_t^{\theta, x, \tilde x_2} 
$$
and as in \eqref{bar Z} and \eqref{bar Z element} it satisfies the SDE
\bae
d \hat Z_t = \left( \left\langle \nabla_x^2 \mu(X_t^{\theta, x}, \theta) \nabla_x X_t^{\theta, x},\  Z_t \right\rangle + \nabla_x \mu(X_t^{\theta, x}, \theta) \hat Z_t\right) dt + \left(\left\langle \nabla_x^2 \sigma(X_t^{\theta, x}, \theta) \nabla_\theta X_t^{\theta, x},\ Z_t \right\rangle + \nabla_x \sigma(X_t^{\theta, x}, \theta) \hat Z_t\right) dW_t,
\eae
Similarly, we can derive the differential inequality
$$
\frac{d}{dt} \e |\hat Z_t|^2 \le -\beta \e|\hat Z_t|^2 + C\e |Z_t|^2.
$$
Consequently,
\beq
\label{tilde x lip}
\e \left|\nabla_x \tilde X_t^{\theta, x, \tilde x_1} - \nabla_x \tilde X_t^{\theta, x, \tilde x_2}\right|^2 \le C e^{-\frac\beta2 t} \left|\tilde x_1 - \tilde x_2\right|^2.
\eeq
Due to Lemma \ref{difference lemma}, 
\beq
\e\left| X_t^{\theta, x_1} - X_t^{\theta, x_2}  \right|^2 \le e^{-\beta t} |x_1 - x_2|^2,
\eeq
which, combined with the dominated convergence theorem, yields
\beq
\label{gradient x bound}
\e\left| \nabla_x X_t^{\theta, x} \right|^2 \le e^{-\beta t}.
\eeq
Therefore,
\bae
&\left| \nabla_x v^{1,2}(x, \tilde x, \theta) \right| \\
\le& C \int_0^\infty  \e \left| \nabla f(X_t^{\theta, x}) \left(\nabla_x \tilde X_t^{\theta, x, \tilde x} - \nabla_x \tilde X_t^{\theta, x, 0} \right) \right| dt + C \int_0^\infty  \e \left|\left\langle \nabla^2 f(X_t^{\theta, x}) \nabla_x X_t^{\theta, x},\  \tilde X_t^{\theta, x, \tilde x} - \tilde X_t^{\theta, x, 0} \right\rangle \right| dt \\
\overset{(a)}{\le}& \int_0^\infty Ce^{-\beta t}|\tilde x|dt + \int_0^\infty Ce^{-\frac\beta2 t}|\tilde x|dt \\
\le& C |\tilde x|,
\eae
where step $(a)$ is by Cauchy-Schwarz inequality, \eqref{gradient x bound}, \eqref{tilde lip} and \eqref{tilde x lip}. The bound for $\nabla_x^2 v^{1,2}$ follows from exactly the same method. Combining the bounds for $v^{1,1}$ and $v^{1,2}$ proves the desired bound \eqref{control v1}. Using the same calculations as in Lemma 3.3 of \cite{wang2022continuous}, we can show that $v^1$ is the classical solution of PDE \eqref{PDE} and thus the proof is completed. 
\end{proof}

We will also need bounds on the moments of $X_t$ and $\tilde X_t$ in order to analyze the fluctuation term $\Delta^i_{\tau_n, \sigma_n + \eta}$. 
\begin{lemma}
	\label{moment}
	There exists a constant $C > 0$ such that the processes $X_t, \tilde X_t$ in \eqref{nonlinear update} satisfy
	\beq
	\label{moment bound}
	\e_x |X_t|^8 \le C \left( 1 + |x|^8 \right), \quad \e_{x,\tilde x} |\tilde X_t|^8 \le C\left( 1 + |\tilde x|^8 \right),
	\eeq 
	where $\e_{x,\tilde x}$ is the conditional expectation given that $X_0 = x$ and $\tilde X_0 = \tilde x$. Furthermore, we have the bounds
	\begin{eqnarray}
	\e_x \left( \sup\limits_{0 \le t' \le t} |X_{t'}|^4 \right) &= O(\sqrt t) \quad \text{as}\  t \to \infty, \label{uniform moment bound 1} \\
	\e_{x, \tilde x} \left( \sup\limits_{0 \le t' \le t} |\tilde X_{t'}|^4 \right) &= O(\sqrt t) \quad \text{as}\ t \to \infty. \label{uniform moment bound 2}	\end{eqnarray}
\end{lemma}
\begin{proof}

By It\^{o}'s formula, for any $m\ge1$ we have 
\bae
\label{ito formula}
d|X_t|^{2m} &= 2m |X_t|^{2m-2} \langle \mu(X_t, \theta_t),\ X_t \rangle dt + 2m |X_t|^{2m-2} \langle \sigma(X_t, \theta_t), \ X_t \rangle dW_t \\
&+ m |X_t|^{2m-2} \cdot |\sigma(X_t, \theta_t)|^2 dt + 2m(m-1) |X_t|^{2m-4} \cdot |\langle X_t,\  \sigma(X_t, \theta_t) \rangle|^2.
\eae
We use induction to prove the bound for the 8-th moment. First let $m=1$ in \eqref{ito formula}, by the same proof as in Lemma \ref{moment stable}, we have
\beq
\frac{d}{dt} \e_x \left| X_t \right|^2 \le -\beta \e_x \left| X_t \right|^2 + \e_x \left( \frac1\beta \left| \mu(0, \theta_t) \right|^2 + 2\left| \sigma(0,\theta_t) \right|^2 \right) \le -\beta \e_x \left| X_t \right|^2 + C,
\eeq
which yields the bound for the second moment
\beq
\label{second moment bound x}
\e_x |X_t|^2 \le C \left( 1 + |x|^2 \right).
\eeq 
For $k \in \{1,2,\cdots, \ell\}$, let $\tilde X_t^{:, k}$ denote the k-th column of $\tilde X_t$. $\left|\tilde X_t^{:, k}\right|^2$ satisfies the following SDE: 
\bae
\label{tilde ito}
d \left|\tilde X^{:,k}_t\right|^2 =& 2\left \langle \nabla_x \mu( X_t, \theta_t ) \tilde X^{:, k}_t +\frac{\partial \mu(X_t, \theta_t)}{\partial \theta_k}, \ \tilde X^{:, k}_t \right\rangle dt + \left| \nabla_x \sigma( X_t, \theta_t) \tilde X^{:,k}_t + \frac{\partial \sigma(X_t, \theta_t)}{\partial \theta_k} \right|^2 dt \\
+& 2\left\langle \nabla_x \sigma( X_t, \theta_t )\tilde X^{:,k}_t + \frac{\partial \sigma(X_t, \theta_t)}{\partial \theta_k}, \ \tilde X^{:,k}_t \right\rangle dW_t.
\eae
Similar to \eqref{grad theta cal}, we can derive the differential inequality
\beq
\label{tilde dissiptive}
\frac{d}{dt} \e_{x, \tilde x} \left| \tilde X^{:,k}_t \right|^2 \le -\beta \e_{x, \tilde x} \left| \tilde X^{:,k}_t \right|^2 + \e_{x,\tilde x} \left( \frac1\beta \left| \frac{\partial \mu(X_t, \theta_t)}{\partial \theta_k}\right|^2 + 2 \left| \frac{\partial \sigma(X_t, \theta_t)}{\partial \theta_k} \right|^2  \right) \le -\beta \e_{x, \tilde x} \left| \tilde X^{:, k}_t \right|^2 + C.
\eeq
Therefore,
\beq
\label{second moment bound tilde x}
\quad \e_{x,\tilde x} |\tilde X_t|^2 \le C\left( 1 + |\tilde x|^2 \right).
\eeq

Now let $m=2$ in \eqref{ito formula} and use the bound \eqref{4 moment dis},
\bae
\frac{d}{dt} \e_x |X_t|^4 &= 4 \e_x \left( |X_t|^2 \left\langle \mu(X_t, \theta_t) ,\ X_t \right\rangle \right) dt + \e_x \left( 2|X_t|^2 |\sigma(X_t, \theta_t)|^2 + 4 \left| \langle \sigma(X_t, \theta_t),\  X_t \rangle \right|^2 \right) \\
&\le \e_x \left[ |X_t|^2 \left( 4\langle \mu(X_t, \theta_t),\ X_t \rangle + 6|\sigma(X_t, \theta_t)|^2 \right) \right] \\
&\le -\beta \e_x |X_t|^4 + C\e_x |X_t|^2,
\eae
which together with \eqref{second moment bound x} and Gronwall's inequality prove the bound for fourth moment of $X_t$. Similarly, as in \eqref{trick}
\bae
\frac{d}{dt} \e_{x,\tilde x} \left|\tilde X^{:,k}_t\right|^4 &\le \e_{x, \tilde x} \left[ \left|\tilde X^{:,k}_t\right|^2 \left( 4\left\langle \nabla_x \mu(X_t, \theta_t)\tilde X^{:,k}_t + \frac{ \partial \mu(X_t, \theta_t)}{\partial \theta_k}, \ \tilde X^{:,k}_t \right\rangle + 6\left|\nabla_x \sigma(X_t, \theta_t)\tilde X^{:,k}_t + \frac{\partial \sigma(X_t, \theta_t)}{\partial \theta_k}\right|^2 \right) \right] \\
&\le -\beta \e_{x, \tilde x} \left|\tilde X^{:,k}_t\right|^4 + C\e_{x, \tilde x} \left|\tilde X^{:,k}_t\right|^2,
\eae
which together with \eqref{second moment bound tilde x} derives the estimate for $\tilde X_t$ in \eqref{moment bound}. By induction, we can prove the bound for the sixth and eighth moments of $(X_t, \tilde X_t)$ in \eqref{moment bound}.

Finally, as in \eqref{ito formula} and use \eqref{trick}, we have 
\bae
\left| X_t \right|^8 &= \left| x \right|^8 +  8\int_0^t |X_s|^6 \left\langle \mu(X_s, \theta_s), \ X_s \right\rangle ds + 8\int_0^t |X_s|^6 \left\langle \sigma(X_s, \theta_s),\ X_s \right\rangle dW_s \\
&+ 24\int_0^t \left|X_t\right|^4 \cdot \left| \left\langle \sigma(X_s, \theta_s), \ X_s \right\rangle \right|^2 ds + 4\int_0^t |X_s|^6 \cdot \left| \sigma(X_s, \theta_s) \right|^2 ds   \\ 
&\le -4\beta \int_0^t \left| X_s \right|^8 ds + C \int_0^t |X_s|^6 ds + 8\int_0^t |X_s|^6 \langle \sigma(X_s, \theta_s),\ X_s \rangle dW_s,
\eae
which together with the Burkholder-Davis-Gundy inequality and \eqref{second moment bound x} derive that there exists a constant $C$ such that
\bae
\label{uniform pre}
\e_x \sup\limits_{0\le t'\le t}\left| X_{t'} \right|^8 &\le |x|^8 + ct + C \e_x \left( \int_0^t |X_s|^{14} \cdot |\sigma(X_s, \theta_s)|^2 ds\right)^{\frac12}\\
&\le |x|^8 + ct + C \e_x \left( \sup_{0\le t'\le t}|X_{t'}|^8 \cdot \int_0^t |X_s|^6 \cdot |\sigma(X_s, \theta_s)|^2 ds\right)^{\frac12}\\
&\overset{(a)}{\le} |x|^8 + ct + \frac12 \e_x \sup\limits_{0\le t'\le t}\left| X_{t'} \right|^8 + C \int_0^t \e_x \left[ |X_s|^6 + |X_s|^8 \right] ds, 
\eae
where step $(a)$ is by Young's inequality. Thus, combining \eqref{moment bound} and \eqref{uniform pre} we obtain 
$$
\e_x \left( \sup\limits_{0 \le t' \le t} |X_{t'}|^8 \right) = O(t) \quad \text{as}\  t \to \infty,
$$
which derives \eqref{uniform moment bound 1}. Similarly for \eqref{uniform moment bound 2}, using It\^{o}'s formula for $\left| \tilde X_t \right|^8$ and the Burkholder-Davis-Gundy inequality,
\bae
\e_{x,\tilde x} \sup\limits_{0\le t'\le t}\left| \tilde X_{t'} \right|^8 &\le |\tilde x|^8 + ct + C \e_{x, \tilde x}\left( \int_0^t |\tilde X_s|^{14} \cdot \left|\nabla_x \sigma(X_s, \theta_s)\tilde X_s + \nabla_\theta \sigma(X_s, \theta_s) \right|^2 ds\right)^{\frac12}\\
&\le |\tilde x|^8 + ct + \frac12 \e_{x, \tilde x} \sup\limits_{0\le t'\le t}\left| \tilde X_{t'} \right|^8 + C \int_0^t \e_{x, \tilde x}\left[ \left|\tilde X_s \right|^6+ \left|\tilde X_s\right|^8 \right] ds,
\eae
which together with \eqref{moment bound} derive \eqref{uniform moment bound 2}.

\end{proof}

We can now bound the first fluctuation term $\Delta^1_{\tau_k, \sigma_k + \eta}$ in \eqref{error integral} using the estimates from Lemma \ref{poisson eq} and Lemma \ref{moment}.
\begin{lemma}
	\label{fluctuation 1}
	Under Assumptions A\ref{dissipative} - A\ref{lr}, for any fixed $\eta>0$,
	\beq
	\label{errors conv}
	\left|\Delta^1_{\tau_n, \sigma_n + \eta}\right| \rightarrow 0 \text { as } n \rightarrow \infty, \quad \text{a.s.}
	\eeq
\end{lemma}
\begin{proof} 
	We will express $\Delta^i_{\tau_n, \sigma_n + \eta}$ in terms of the Poisson equation in Lemma \ref{poisson eq} and then prove it vanishes as $n$ becomes large. Consider the function 
	$$
	G^1(x,\tilde x, \theta) = (\e_{\pi_\theta} f(Y)- \beta) \left( \nabla f(x)\tilde x - \nabla_\theta \e_{\pi_\theta}f(Y) \right)^\top.
	$$
	By Lemma \ref{poisson eq}, the Poisson equation $ \mathcal{L}^\theta_{x\tilde x} u(x,\tilde x, \theta) = G^1(x, \tilde x, \theta) $ will have a unique smooth solution $v^1(x, \tilde x, \theta)$ that grows at most linearly in $(x, \tilde x)$. Let us apply Itô's formula to the function
	$$
	u^1(t, x, \tilde x, \theta) := \alpha_{t} v^1(x, \tilde x, \theta) \in \mathbb{R}^\ell,
	$$
	evaluated on the stochastic process $(X_t, \tilde X_t, \theta_t)$. Recall that $u_i$ denotes the $i$-th element of $u$ and $\tilde X_t^{:,k} $ be the k-th column of the matrix $\tilde X_t$ for $i,k \in \{1,2,\cdots, \ell\}$. Then,
	\bae
	u^1_i\left(\sigma, X_{\sigma}, \tilde X_\sigma, \theta_{\sigma}\right) =& u^1_i\left(\tau, X_{\tau}, \tilde X_\tau, \theta_{\tau}\right) + \int_{\tau}^{\sigma} \partial_{s} u^1_i\left(s, X_{s}, \tilde X_s, \theta_{s}\right) ds + \int_{\tau}^{\sigma} \mathcal{L}^{\theta_s}_{x\tilde x} u^1_i\left(s, X_{s}, \tilde X_s, \theta_{s}\right) ds \\
	+& \int_{\tau}^{\sigma} \nabla_\theta u^1_i\left(s, X_{s}, \tilde X_s, \theta_{s}\right) d\theta_s  + \int_{\tau}^{\sigma} \nabla_{x} u^1_i\left(s, X_{s}, \tilde X_s, \theta_{s}\right) \sigma(X_s, \theta_s) dW_{s} \\
	+&  \sum_{k=1}^\ell \int_{\tau}^{\sigma} \nabla_{\tilde x^{:,k}} u^1_i\left(s, X_{s}, \tilde X_s, \theta_{s}\right) \left( \nabla_x \sigma(X_s, \theta_s) \tilde X_s^{:,k} + \frac{\partial \sigma(X_s, \theta_s)}{\partial \theta_k} \right)dW_{s}.
	\eae
	Rearranging the previous equation, we obtain the representation
	\bae
	\label{representation 1}
	\Delta^1_{\tau_n,\sigma_n + \eta} =& \int_{\tau_{n}}^{\sigma_{n}+ \eta} \alpha_{s} G^1(X_s, \tilde X_s, \theta_s) ds = \int_{\tau_{k}}^{\sigma_{k}+\eta} \mathcal{L}^{\theta_s}_{x\tilde x} u^1\left(s, X_{s}, \tilde X_s, \theta_{s}\right) ds \\
	=& \alpha_{\sigma_{n}+ \eta} v^1\left(X_{\sigma_{n}+\eta}, \tilde  X_{\sigma_{n}+\eta},  \theta_{\sigma_{n}+\eta}\right)-\alpha_{\tau_{n}} v^1\left(X_{\tau_{n}}, \tilde X_{\tau_{n}}, \theta_{\tau_{n}}\right)-\int_{\tau_{n}}^{\sigma_{n}+\eta} \alpha'_{s} v^1\left(X_{s}, \tilde X_s,  \theta_{s}\right) ds  \\
	+&\int_{\tau_{n}}^{\sigma_{n}+\eta} 2\alpha^2_{s} \nabla_\theta v^1\left(X_{s}, \tilde X_s, \theta_{s}\right) (f(\bar X_s) - \beta) \left( \nabla f(X_s) \tilde X_s\right)^\top ds - \int_{\tau_{n}}^{\sigma_{n}+\eta} \alpha_s \nabla_{x} v^1\left(X_{s}, \tilde X_s, \theta_{s}\right) \sigma(X_s, \theta_s) dW_{s}\\
	-& \sum_{k=1}^\ell \int_{\tau_n}^{\sigma_n +\eta} \alpha_s \nabla_{\tilde x^{:,k}} v^1\left(X_{s}, \tilde X_s, \theta_{s}\right) \left( \nabla_x \sigma(X_s, \theta_s) \tilde X_s^{:,k} + \frac{\partial \sigma(X_s, \theta_s)}{\partial \theta_k} \right)dW_{s}.
	\eae
	The next step is to treat each term on the right hand side of \eqref{representation 1} separately. For this purpose, let us first set
	\beq
	J_{t}^{1,1}=\alpha_{t} \sup _{s \in[0, t]}\left|v^1\left(X_{s}, \tilde X_s, \theta_{s}\right)\right|.
	\eeq
	By \eqref{control v1} and Lemma \ref{moment}, there exists a constant $C$ such that 
	\bae
		\e\left|J_{t}^{1,1}\right|^{2} &\leq C \alpha_{t}^{2} \e \left[1 + \sup _{s \in[0, t]}\left|X_{s}\right|^2 + \sup _{s \in[0, t]}\left|\tilde X_{s}\right|^2 \right]\\
		&= C \alpha_{t}^{2}\left[1+\sqrt{t} \frac{\e \sup _{s \in[0, t]}\left|X_{s}\right|^2 + \e \sup _{s \in[0, t]}\left|\tilde X_{s}\right|^2 }{\sqrt{t}} \right] \\
		&\leq C \alpha_{t}^{2} \sqrt{t}.
	\eae
	Let $p>0$ be the constant in Assumption A\ref{lr} such that $\lim _{t \rightarrow \infty} \alpha_{t}^{2} t^{1 / 2+2 p}=0$ and for any $\delta \in(0, p)$ define the event $A_{t, \delta}=\left\{J_{t}^{1,1} \geq t^{\delta-p}\right\} .$ Then we have for $t$ large enough such that $\alpha_{t}^{2} t^{1 / 2+2 p} \leq 1$
	$$
	\prob \left(A_{t, \delta}\right) \leq \frac{\e\left|J_{t}^{1,1}\right|^{2}}{t^{2(\delta-p)}} \leq C \frac{\alpha_{t}^{2} t^{1 / 2+2 p}}{t^{2 \delta}} \leq C \frac{1}{t^{2 \delta}}.
	$$
	The latter implies that
	$$
	\sum_{m \in \mathbb{N}} \prob \left(A_{2^{m}, \delta}\right)<\infty.
	$$
	Therefore, by the Borel-Cantelli lemma we have that for every $\delta \in(0, p)$ there is a finite positive random variable $d(\omega)$ and some $m_{0}<\infty$ such that for every $m \geq m_{0}$ one has
	$$
	J_{2^{n}}^{1, 1} \leq \frac{d(\omega)}{2^{m(p-\delta)}}.
	$$
	Thus, for $t \in\left[2^{m}, 2^{m+1}\right)$ and $m \geq m_{0}$ one has for some finite constant $C<\infty$
	$$
	J_{t}^{1, 1} \leq C \alpha_{2^{m+1}} \sup _{s \in\left(0,2^{m+1}\right]}\left| v^1\left(X_{s}, \tilde X_s, \theta_{s}\right) \right| \leq C \frac{d(\omega)}{2^{(m+1)(p-\delta)}} \leq C \frac{d(\omega)}{t^{p-\delta}},
	$$
	which proves that for $t \geq 2^{m_{0}}$ with probability one
	\beq
	\label{conv 1}
	J_{t}^{1, 1} \leq C \frac{d(\omega)}{t^{p-\delta}} \rightarrow 0, \text { as } t \rightarrow \infty.
	\eeq
	
	Next we consider the term
	$$
	J_{t, 0}^{1,2} = \int_{0}^{t}\left|\alpha_{s}^{\prime} v^1\left(X_{s}, \tilde X_s, \theta_{s}\right) - 2\alpha^2_{s} \nabla_\theta v^1\left(X_{s}, \tilde X_s, \theta_{s}\right) (f(\bar X_s) - \beta) \left( \nabla f(X_s) \tilde X_s\right)^\top \right| ds.
	$$
	Noting that by the same approach for $X_t$ in Lemma \ref{moment}, we can prove that there exists a constant $C>0$ such that
	\beq
	\label{moment bar}
	\e_{\bar x} |\bar X_t|^4 \le C \left( 1 + |\bar x|^4 \right), \quad \e_{\bar x} \left( \sup\limits_{0 \le t' \le t} |\bar X_{t'}|^2 \right) = O(\sqrt t) \quad \text{as}\  t \to \infty.
	\eeq
	Thus 
	$$
	\begin{aligned}
		\sup _{t>0} \e \left|J_{t, 0}^{1,2}\right| & \overset{(a)}{\le} C \int_{0}^{\infty}\left(\left|\alpha_{s}^{\prime}\right|+\alpha_{s}^{2}\right)\left(1  + \e\left|X_{s}\right|^4 + \e|\tilde X_s|^4 + \e\left|\bar X_{s}\right|^2 \right) ds \\
		& \overset{(b)}{\le} C \int_{0}^{\infty}\left(\left|\alpha_{s}^{\prime}\right|+\alpha_{s}^{2}\right) d s \\
		& \leq C,
	\end{aligned}
	$$
	where step $(a)$ is by Assumption A\ref{f} and \eqref{control v1} and in step $(b)$ we use \eqref{moment bound}. Thus  there is a finite random variable $J_{\infty, 0}^{1, 2}$ such that
	\beq
	\label{conv 2}
	J_{t, 0}^{1, 2} \rightarrow J_{\infty, 0}^{1, 2}, \text{as} \ t \rightarrow \infty \ \text{with probability one}.
	\eeq 
	
	The last term we need to consider is the martingale term
$$
	\begin{aligned}
	J_{t, 0}^{1, 3} 
	&=\int_{0}^{t} \alpha_s  \nabla_{x} v^1\left(X_{s}, \tilde X_s, \theta_{s}\right) \sigma(X_s, \theta_s) dW_{s} \\
	&+ \sum_{k=1}^\ell \int_{0}^{t} \alpha_s \nabla_{\tilde x^{:,k}} v^1\left(X_{s}, \tilde X_s, \theta_{s}\right) \left( \nabla_x \sigma(X_s, \theta_s) \tilde X_s^{:,k} + \frac{\partial \sigma(X_s, \theta_s)}{\partial \theta_k} \right)dW_{s}.
	\end{aligned}
	$$
	By Doob's inequality, Assumption A\ref{lr}, \eqref{control v1}, \eqref{moment bound}, and using calculations similar to the ones for the term $J_{t, 0}^{1, 2}$, we can show that for some finite constant $C<\infty$,
	$$
	\sup _{t>0} \e \left|J_{t, 0}^{1, 3} \right|^{2} \le C \int_{0}^{\infty} \alpha_{s}^{2}\left(1 + \e|X_t|^4 + \e|\tilde X_t|^4 \right) d s<\infty
	$$
	Thus, by Doob's martingale convergence theorem there is a square integrable random variable $J_{\infty, 0}^{1,3}$ such that
	\beq
	\label{conv 3}
	J_{t, 0}^{1, 3} \to J_{\infty, 0}^{1, 3},\quad  \text{as} \ t \to \infty \ \text{ both almost surely and in $L^{2}$}.
	\eeq
	Let us now return to \eqref{representation 1}. Using the terms $J_{t}^{1,1}, J_{t, 0}^{1,2}$, and $J_{t, 0}^{1,3}$ we can write
	$$
	\left|\Delta^1_{\tau_n, \sigma_n + \eta}\right| \leq J_{\sigma_{n}+\eta}^{1,1}+J_{\tau_{n}}^{1,1}+\left|J_{\sigma_{n}+\eta, \tau_{n}}^{1,2}\right| + \left|J_{\sigma_{n}+\eta, \tau_{n}}^{1,3}\right|,
	$$
	which together with \eqref{conv 1}, \eqref{conv 2}, and \eqref{conv 3} prove the statement of the Lemma.
\end{proof}

We will next prove a similar convergence result for $\Delta^2_{\tau_n, \sigma_n + \eta}$. We must first prove an extension of Lemma \ref{poisson eq} for the Poisson equation.
\begin{lemma}
	\label{poisson eq 2}
	Define the error function 
	\beq
	\label{function 2}
	G^2(x,\tilde x, \bar x, \theta) = [ f(\bar x) -  \e_{\pi_{\theta}}f(Y) ] \left( \nabla f(x) \tilde x \right)^\top.
	\eeq
	Under Assumptions (A\ref{dissipative}) - (A\ref{f}), the function 
	\beq
	\label{representation2}
	v^2(x, \tilde x, \bar x, \theta) =  -\int_0^\infty \e G^2(X_t^{\theta,x},\tilde X_t^{\theta, x, \tilde x}, \bar X_t^{\theta, \bar x}, \theta) dt 
	\eeq
	is the classical solution of the Poisson equation 
	\beq
	\label{PDE 2}
	\mathcal{L}^\theta_{x, \tilde x, \bar x} u(x, \tilde x, \bar x, \theta) = G^2(x, \tilde x, \bar x, \theta),
	\eeq
	where $\mathcal{L}^\theta_{x,\tilde x,\bar x}$ is generator of the process $(X_\cdot^{\theta, x}, \tilde X_\cdot^{\theta, x, \tilde x}, \bar X_\cdot^{\theta, \bar x})$, i.e. for any test function $\varphi$
	\beq
	\mathcal{L}^\theta_{x,\tilde x,\bar x} \varphi(x,\tilde x, \bar x) = \mathcal{L}_{x,\tilde x}^\theta \varphi(x, \tilde x, \bar x) + \mathcal{L}_{\bar x}^\theta \varphi(x, \tilde x, \bar x).
	\eeq 
	Furthermore, this solution satisfies the bound 
	\bae
	\label{control v2}
	\left| v^2(x, \tilde x, \bar x, \theta)\right| + \left| \nabla_{\bar x} v^2(x, \tilde x, \bar x, \theta)\right| + \left|\nabla_\theta v^2(x, \tilde x, \bar x, \theta)\right| + \left| \nabla_x v^2(x, \tilde x, \bar x, \theta)\right| + \left| \nabla_{\tilde x} v^2(x, \tilde x, \bar x, \theta)\right| \le C\left(1 + |\bar x|\right)\left( 1 + |\tilde x| \right),
	\eae
	where $C$ is a constant independent of $(x,\tilde x, \bar x, \theta)$.
\end{lemma} 

\begin{proof}
	The proof is exactly the same as in Lemma \ref{poisson eq} except for the presence of the dimension $\bar x$ and $\mathcal{L}_{\bar x}$.  Since $X^\theta_t$ and $\bar X^\theta_t$ are i.i.d., the bounds from Proposition \ref{ergodic estimation} are also true for $\bar X_t$. We first show that the integral \eqref{representation2} is finite. Note that
	\begin{eqnarray}
	v^2(x,\tilde x, \bar x, \theta) &=& \int_0^\infty  
	\e\left[ \left(\e_{\pi_{\theta}}f(Y) - f(\bar X_t^{\theta, \bar x}) \right) \cdot \left(\nabla f(X_t^{\theta, x}) \tilde X_t^{\theta, x, \tilde x}\right)^\top \right] dt \notag \\
	&\overset{(a)}{=}& \int_0^\infty  \left( \e_{\pi_{\theta}}f(Y) - \e f(\bar X_t^{\theta, \bar x}) \right) \cdot \e \left[ \nabla f(X_t^{\theta, x}) \tilde X_t^{\theta, x, \tilde x} \right]^\top dt,
	\end{eqnarray}
	where step $(a)$ is due to the independence of $\bar X^{\theta, \bar x}_\cdot$ and $(X^{\theta, x}_\cdot, \tilde X^{\theta, x, \tilde x}_\cdot)$. As in \eqref{tilde ito} and \eqref{tilde dissiptive}, we can prove 
	$$
	\e \left|\tilde X_t^{\theta, x, \tilde x}\right|^2 \le C \left(1+|\tilde x|^2\right) 
	$$
	and thus by Assumption A\ref{f}
	\beq
	\label{expectation bound}
	\left| \e \left[ \nabla f(X_t^{\theta, x}) \tilde X_t^{\theta, x, \tilde x} \right] \right| \le C \e  \left| \tilde X_t^{\theta, x, \tilde x} \right| \le C\left( 1 + |\tilde x|\right),
	\eeq
    which together with Proposition \ref{ergodic estimation} yields
	\begin{eqnarray}
	\left| v^2(x,\tilde x, \bar x, \theta) \right| \le C \left(1 + \left| \tilde x \right| \right) \cdot \int_0^\infty  \left| \e f(\bar X_t^{\theta, \bar x}) - \e_{\pi_{\theta}}f(Y) \right| dt \le C\left(1 + |\tilde x| \right) \left(1 + |\bar x|\right).
	\end{eqnarray}
	
	We next show that $v^2(x,\tilde x, \bar x, \theta)$ is differentiable with respect to $(x, \tilde x, \bar x, \theta)$. Similar to Lemma \ref{poisson eq}, we first change the order of differentiation and integration and show the corresponding integral exists. Then, we apply DCT to prove that the differentiation and integration can be interchanged. For the ergodic process $\bar X_\cdot^\theta$, by Proposition \ref{ergodic estimation} and \eqref{expectation bound}, we have the following bound for $i = 1,2$:
	\beq
	\left|\nabla^i_{\bar x} v^2(x,\tilde x, \bar x, \theta)\right| \le \int_0^\infty \left| \nabla^i_{\bar x} \e f(\bar X_t^{\theta, \bar x})\right| \cdot \left| \e \left[ \nabla f(X_t^{\theta, x}) \tilde X_t^{\theta, x, \tilde x} \right]  \right|dt \le  C\left(1 + |\tilde x| \right).
	\eeq
	Note that 
	\begin{eqnarray}
	\nabla_\theta X_t^\theta = \tilde X_t^{\theta, x, 0}
	\end{eqnarray}
	and thus, as in the proof of Proposition of \ref{ergodic estimation} and Lemma \ref{poisson eq}, it is easy to prove the bounds
	\bae
    \label{x tilde x bound}
    &\sup_{\theta \in \mathbb{R}^\ell, x \in \mathbb{R}^d} \left| \nabla^i_x \tilde X_t^{\theta, x, \tilde x} \right|^2 \le Ce^{-\beta t}, \quad \sup_{\theta \in \mathbb{R}^\ell, x \in \mathbb{R}^d} \left| \nabla^i_{\tilde x} \tilde X_t^{\theta, x, \tilde x} \right|^2 \le Ce^{-\beta t}, \quad i=1,2, \\
    &\sup_{\theta \in \mathbb{R}^\ell, x \in \mathbb{R}^d} \left| \nabla_\theta \tilde X_t^{\theta, x, \tilde x} \right|^2 \le C, \quad \sup_{\theta \in \mathbb{R}^\ell, x \in \mathbb{R}^d} \left| \nabla_x \nabla_{\tilde x} \tilde X_t^{\theta, x, \tilde x} \right|^2 \le Ce^{-\beta t},
	\eae
	which derives 
	\bae
	\label{second term bound}
	\sum_{i=1}^2 \left| \nabla^i_x \e \left[\nabla f(X_t^{\theta, x}) \tilde X_t^{\theta, x, \tilde x} \right] \right| + \left| \nabla_\theta \e \left[\nabla f(X_t^{\theta, x}) \tilde X_t^{\theta, x, \tilde x} \right] \right| &\le C\left(1+\left|\tilde x\right|\right),\\
	\sum_{i=1}^2 \left| \nabla^i_{\tilde x} \e \left[\nabla f(X_t^{\theta, x}) \tilde X_t^{\theta, x, \tilde x} \right] \right| + \left| \nabla_x \nabla_{\tilde x} \e \left[\nabla f(X_t^{\theta, x}) \tilde X_t^{\theta, x, \tilde x} \right] \right| &\le C.
	\eae
	Therefore, for $i = 1,2$,
	\bae
	\left| \nabla_x^i v^2(x, \tilde x, \bar x, \theta) \right| \le \int_0^\infty \left| \e_{\pi_{\theta}}f(Y) - \e f(\bar X_t^{\theta, \bar x}) \right| \cdot \left| \nabla^i_x \e \left[ \nabla f(X_t^\theta) \tilde X_t^{\theta, x, \tilde x}\right] \right| dt &\le C\left(1 + |\bar x| \right) \left(1 + |\tilde x|\right), \\
	\left| \nabla_{\tilde x}^i v^2(x, \tilde x, \bar x, \theta) \right| \le \int_0^\infty \left| \e_{\pi_{\theta}}f(Y) - \e f(\bar X_t^{\theta, \bar x}) \right| \cdot \left| \nabla^i_{\tilde x} \e \left[ \nabla f(X_t^\theta) \tilde X_t^{\theta, x, \tilde x}\right] \right| dt &\le C\left(1 + |\bar x| \right)
 	\eae
     and
	\begin{equation}
	\begin{aligned}
	\left| \nabla_\theta v^2(x, \tilde x, \bar x, \theta) \right| =& \left| \int_0^\infty \nabla_\theta \left( \left[ \e_{\pi_{\theta}}f(Y) - \e f(\bar X_t^{\theta,\bar x}) \right] \cdot \e \left[ \nabla f(X_t^\theta) \tilde X_t^{\theta, x, \tilde x} \right] \right) dt \right| \\
	\le& \left| \int_0^\infty  \left[ \nabla_\theta \e_{\pi_{\theta}}f(Y) - \nabla_\theta \e f(\bar X_t^{\theta,\bar x}) \right] \cdot \e \left[ \nabla f(X_t^\theta) \tilde X_t^{\theta, x, \tilde x} \right]  dt \right|\\
	+& \left| \int_0^\infty  \left[ \e_{\pi_{\theta}}f(Y) - \e f(\bar X_t^{\theta,\bar x}) \right] \cdot \nabla_\theta \e \left[ \nabla f(X_t^\theta) \tilde X_t^{\theta, x, \tilde x} \right] dt \right|\\
	\le& C\left(1 + |\bar x| \right) \left( 1+ |\tilde x|\right).
	\end{aligned}
	\end{equation}
	Finally,
	\beq
	\left| \nabla_x \nabla_{\tilde x}v^2(x,\tilde x, \bar x, \theta) \right| \le
	\int_0^\infty  \left| \e_{\pi_{\theta}}f(Y) - \e f(\bar X_t^{\theta, \bar x}) \right| \cdot \left| \nabla_x \nabla_{\tilde x} \e \left[\nabla f(X_t^{\theta, x}) \tilde X_t^{\theta, x, \tilde x} \right] \right| dt \le C\left(1 + |\bar x| \right).
	\eeq
	By the same calculations as in Lemma 3.3 of \cite{wang2022continuous}, it can be shown that $v^2$ is the classical solution of PDE \eqref{PDE 2} and the bound \eqref{control v2} holds.
\end{proof}

Now we can bound the second fluctuation term $Z_t^2$. The proof is exactly the same as in Lemma \ref{fluctuation 1}.
\begin{lemma}
	\label{fluctuation 2}
	Under Assumptions (A\ref{dissipative}) - (A\ref{lr}), for any fixed $\eta>0$,
	\beq
	\label{errors conv 2}
	\left|\Delta^2_{\tau_n, \sigma_n + \eta}\right| \rightarrow 0, \text { as } n \rightarrow \infty, \quad \text{a.s.}
	\eeq
\end{lemma}

\begin{proof}
	Consider the function 
	\beq
	G^2(x,\tilde x, \bar x, \theta) = \left[ f(\bar x) -  \e_{\pi_{\theta}}f(Y)\right] \left(\nabla f(x) \tilde x \right)^\top.
	\eeq
	Let $v^2$ be the solution of \eqref{PDE 2} in Lemma \ref{poisson eq 2}. We apply Itô formula to the function $u^2(t, x, \tilde x, \bar x, \theta)=\alpha_{t} v^2(x, \tilde x, \bar x, \theta)$
	evaluated on the stochastic process $(X_t, \tilde X_t, \bar X_t, \theta_t)$ and get for any $i \in \{1,2,\cdots, \ell\}$ 
	\bae
		&u^2_i\left(\sigma, X_{\sigma}, \tilde X_{\sigma}, \bar X_\sigma,\theta_{\sigma}\right)-u^2_i\left(\tau, X_{\tau}, \tilde X_\tau, \bar X_\tau, \theta_{\tau}\right) \\
		=& \int_{\tau}^{\sigma} \partial_{s} u^2_i\left(s, X_{s}, \tilde X_s, \bar X_s, \theta_{s}\right) ds + \int_{\tau}^{\sigma} \mathcal{L}^{\theta_s}_{x,\tilde x} u^2_i\left(s, X_{s}, \tilde X_s, \bar X_s, \theta_{s}\right) ds + \int_{\tau}^{\sigma} \mathcal{L}^{\theta_s}_{\bar x} u^2_i\left(s, X_{s}, \tilde X_s, \bar X_s, \theta_{s}\right) ds\\
		+& \int_{\tau}^{\sigma} \nabla_{\theta} u^2_i\left(s, X_{s}, \tilde X_s, \bar X_s, \theta_{s}\right) d\theta_s + \int_{\tau}^{\sigma} \nabla_{x} u_i^2\left(s, X_{s}, \tilde X_s, \bar X_s, \theta_{s}\right) dW_{s}  + \int_{\tau}^{\sigma} \nabla_{\bar x} u^2_i\left(s, X_{s}, \tilde X_s, \bar X_s, \theta_{s}\right) d\bar W_{s} \\
		+& \sum_{k=1}^\ell \int_{\tau}^{\sigma} \nabla_{\tilde x^{:,k}} u^2_i\left(s, X_{s}, \tilde X_s, \bar X_s, \theta_{s}\right) \left( \nabla_x \sigma(X_s, \theta_s) \tilde X_s^{:,k} + \frac{\partial \sigma(X_s, \theta_s)}{\partial \theta_k} \right)dW_{s}.
	\eae
	Rearranging the previous equation, we obtain the representation
	\bae
	\label{representation 2}
	&\Delta^2_{\tau_n,\sigma_n + \eta} = \int_{\tau_{n}}^{\sigma_{n}+ \eta} \alpha_{s} G^2( X_s, \tilde X_s, \bar X_s, \theta_s) ds = \int_{\tau_{n}}^{\sigma_{n}+\eta} \mathcal{L}^{\theta_s}_{x, \tilde x, \bar x} u^2\left(s, X_{s}, \tilde X_s, \bar X_s, \theta_{s}\right) ds \\
	=& \alpha_{\sigma_{n}+ \eta} v^2\left( X_{\sigma_{n}+\eta}, \tilde X_{\sigma_n+ \eta}, \bar X_{\sigma_n+ \eta}, \theta_{\sigma_{n}+\eta}\right)-\alpha_{\tau_{n}} v^2\left(X_{\tau_{n}}, \tilde X_{\tau_{n}}, \bar X_{\tau_n}, \theta_{\tau_{n}}\right)-\int_{\tau_{n}}^{\sigma_{n}+\eta} \alpha'_{s} v^2\left(X_{s}, \tilde X_s, \bar X_s, \theta_{s}\right) ds \\
	-& \int_{\tau_{n}}^{\sigma_{n}+\eta} \alpha_s \nabla_{\bar x} v^2\left(X_{s}, \tilde X_s, \bar X_s, \theta_{s}\right) d\bar W_{s} + \int_{\tau_{n}}^{\sigma_{n}+\eta} 2\alpha^2_{s} \nabla_\theta v^2\left(X_{s}, \tilde X_s, \bar X_s, \theta_{s}\right) \left(f(\bar X_s) - \beta\right) \left( \nabla f(X_s) \tilde X_s\right)^\top ds \\
	-& \int_{\tau_{n}}^{\sigma_{n}+\eta} \alpha_s \nabla_{x} v^2\left(X_{s}, \tilde X_s, \bar X_s, \theta_{s}\right) dW_s - \sum_{k=1}^\ell \int_{\tau_n}^{\sigma_n+\eta} \alpha_s \nabla_{\tilde x^{:,k}} v^2\left(X_{s}, \tilde X_s, \bar X_s, \theta_{s}\right) \left( \nabla_x \sigma(X_s, \theta_s) \tilde X_s^{:,k} + \frac{\partial \sigma(X_s, \theta_s)}{\partial \theta_k} \right)dW_{s}.
	\eae
	
	The next step is to treat each term on the right hand side of \eqref{representation 2} separately. 
	For this purpose, let us first set
	\beq
	J_{t}^{2,1}=\alpha_{t} \sup _{s \in[0, t]}\left|v^2\left(X_{s}, \tilde X_s, \bar X_s, \theta_{s}\right)\right|.
	\eeq
	Combining Lemma \ref{moment}, \eqref{control v2} and \eqref{moment bar}, we know that there exists a constant $C$ such that 
	\bae
	\e \left|J_{t}^{2,1}\right|^{2} &\leq C \alpha_{t}^{2} \e\left[1  + \sup _{s \in[0, t]}\left|\tilde X_{s}\right|^{4} + \sup _{s \in[0, t]}\left|\bar X_{s}\right|^{4} \right]\\
	&= C \alpha_{t}^{2}\left[1+\sqrt{t} \frac{ \e \sup _{s \in[0, t]}\left|\tilde X_{s}\right|^{4} + \e \sup _{s \in[0, t]}\left|\bar X_{s}\right|^{4}}{\sqrt{t}}  \right] \\
	&\leq C \alpha_{t}^{2} \sqrt{t}.
	\eae
	Let $p>0$ be the constant in Assumption A\ref{lr} such that $\displaystyle \lim _{t \rightarrow \infty} \alpha_{t}^{2} t^{1 / 2+2 p}=0$ and for any $\delta \in(0, p)$ define the event $A_{t, \delta}=\left\{J_{t}^{2,1} \geq t^{\delta-p}\right\} .$ Then we have for $t$ large enough such that $\alpha_{t}^{2} t^{1 / 2+2 p} \leq 1$ and 
	$$
	\prob \left(A_{t, \delta}\right) \leq \frac{\e\left|J_{t}^{2,1}\right|^{2}}{t^{2(\delta-p)}} \leq C \frac{\alpha_{t}^{2} t^{1 / 2+2 p}}{t^{2 \delta}} \leq C \frac{1}{t^{2 \delta}}.
	$$
	The latter implies that
	$$
	\sum_{m \in \mathbb{N}} \prob\left(A_{2^{m}, \delta}\right)<\infty.
	$$
	Therefore, by the Borel-Cantelli lemma we have that for every $\delta \in(0, p)$ there is a finite positive random variable $d(\omega)$ and some $m_{0}<\infty$ such that for every $n \geq m_{0}$ one has
	$$
	J_{2^{n}}^{2,1} \leq \frac{d(\omega)}{2^{m(p-\delta)}}.
	$$
	Thus for $t \in\left[2^{m}, 2^{m+1}\right)$ and $m \geq m_{0}$ one has for some finite constant $C<\infty$
	$$
	J_{t}^{2,1} \leq C \alpha_{2^{m+1}} \sup _{s \in\left(0,2^{m+1}\right]}\left| v^2\left(X_{s}, \tilde X_s, \bar X_s, \theta_{s}\right) \right| \leq C \frac{d(\omega)}{2^{(m+1)(p-\delta)}} \leq C \frac{d(\omega)}{t^{p-\delta}},
	$$
	which derives that for $t \geq 2^{m_{0}}$ we have with probability one
	\beq
	\label{conv 1 2}
	J_{t}^{2,1} \leq C \frac{d(\omega)}{t^{p-\delta}} \rightarrow 0, \text { as } t \rightarrow \infty.
	\eeq

	Next we consider the term
	$$
	J_{t, 0}^{2,2}=\int_{0}^{t}\left|\alpha_{s}^{\prime} v^2\left(X_{s}, \tilde X_s, \bar X_s, \theta_{s}\right) - 2\alpha^2_{s} \nabla_\theta v^2\left(X_{s}, \tilde X_s, \bar X_s, \theta_{s}\right) \left(f(\bar X_s) - \beta\right) \left( \nabla f(X_s) \tilde X_s \right)^\top \right| ds
	$$
    and thus we see that there exists a constant $0<C<\infty$ such that 
	$$
	\begin{aligned}
		\sup_{t>0} \e \left|J_{t, 0}^{2,2}\right| & \overset{(a)}{\le} C \int_{0}^{\infty}\left(\left|\alpha_{s}^{\prime}\right|+\alpha_{s}^{2}\right)\left(1 + \e\left|\bar X_{s}\right|^{4} + \e\left|\bar X_{s}\right|^{4}\right) ds \\
		& \overset{(b)}{\le} C \int_{0}^{\infty}\left(\left|\alpha_{s}^{\prime}\right|+\alpha_{s}^{2}\right) d s \\
		& \leq C,
	\end{aligned}
	$$
	where in step $(a)$ we use \eqref{control v2} and in step $(b)$ we use Lemma \ref{moment} and \eqref{moment bar}. Thus we know there is a finite random variable $J_{\infty, 0}^{2,2}$ such that
	\beq
	\label{conv 2 2}
	J_{t, 0}^{2,2} \rightarrow J_{\infty, 0}^{2,2}, \ \text{as} \ t \rightarrow \infty \ \text{with probability one}.
	\eeq 
	
	The last term we need to consider is the martingale term
	$$
	\begin{aligned}
	J_{t, 0}^{2,3}=&\int_{0}^{t} \alpha_s  \nabla_{x} v^2\left(X_{s}, \tilde X_s, \bar X_s, \theta_{s}\right) dW_{s} + \int_{0}^{t} \alpha_s \nabla_{\bar x} v^2\left(X_{s}, \tilde X_s, \bar X_s, \theta_{s}\right) d\bar W_{s}\\
	+&\sum_{k=1}^\ell \int_{0}^{t} \alpha_s \nabla_{\tilde x^{:,k}} v^2 \left(X_{s}, \tilde X_s, \bar X_s, \theta_{s}\right) \left( \nabla_x \sigma(X_s, \theta_s) \tilde X_s^{:,k} + \frac{\partial \sigma(X_s, \theta_s)}{\partial \theta_k} \right)dW_{s}.
	\end{aligned}
	$$
	Notice that Doob's inequality and the bounds of \eqref{control v2} (using calculations similar to the ones for the term $J_{t, 0}^{2,2}$ ) give us that for some finite constant $K<\infty$, we have
	$$
	\sup _{t>0} \e\left|J_{t, 0}^{2,3}\right|^{2} \leq K \int_{0}^{\infty} \alpha_{s}^{2} d s<\infty.
	$$
	Thus, by Doob's martingale convergence theorem there is a square integrable random variable $J_{\infty, 0}^{(3)}$ such that
	\beq
	\label{conv 3 2}
	J_{t, 0}^{2,3} \to J_{\infty, 0}^{2,3}, \quad
	\text{as}\  t \to \infty \ \text{both almost surely and in $L^{2}$}.
	\eeq	
	Let us now go back to \eqref{representation 2}. Using the terms $J_{t}^{2,1}, J_{t, 0}^{2,2}$ and $J_{t, 0}^{2,3}$ we can write
	$$
	\left|\Delta^2_{\tau_n, \sigma_n + \eta}\right| \leq J_{\sigma_{n}+\eta}^{2,1}+J_{\tau_{n}}^{2,1}+J_{\sigma_{n}+\eta, \tau_{n}}^{2,2}+\left|J_{\sigma_{k}+\eta, \tau_{n}}^{2,3}\right|,
	$$
	which together with \eqref{conv 1 2}, \eqref{conv 2 2} and \eqref{conv 3 2} prove the statement of the Lemma.
\end{proof}

By \eqref{invariant density}, we know that
\beq
\left| \nabla_\theta J(\theta) \right| = 2 \left| \e_{\pi_\theta} f(Y) - \beta \right| \cdot \left| \nabla_\theta \e_{\pi_\theta} f(Y) \right| \le C.
\eeq
Therefore, the objective function $J(\theta)$ is Lipschitz continuous with respect to $\theta$. The following lemmas are the same as in \cite{wang2022continuous} and thus we omit the proofs.

\begin{lemma}
	\label{estimation f}
    Under Assumptions (A\ref{dissipative})-(A\ref{lr}), choose $\mu>0$ in \eqref{cycle of time} such that for the given $\kappa>0$, one has $3 \mu + \frac{\mu}{8 \kappa}=\frac{1}{2L_{\nabla J}}$, where $L_{\nabla J}$ is the Lipschitz constant of objective function $J$ in \eqref{objective function}. Then for $n$ large enough and $\eta>0$ small enough (potentially random depending on $n$), one has $\int_{\tau_{n}}^{\sigma_n + \eta} \alpha_{s} d s>\mu$. In addition we also have $\frac{\mu}{2} \leq \int_{\tau_{n}}^{\sigma_{n}} \alpha_{s} d s \leq \mu$ with probability one.
\end{lemma}

\begin{lemma}
	\label{decreasing f}
	Under Assumptions (A\ref{dissipative})-(A\ref{lr}), suppose that there exists an infinite number of intervals $I_n = [\tau_n, \sigma_n)$. Then there is a fixed constant $ \gamma_1 = \gamma_1(\kappa) > 0$ such that for n large enough,
	\beq
	J(\theta_{\sigma_n}) - J( \theta_{\tau_n}) \le -\gamma_1.
	\eeq
\end{lemma}

\begin{lemma}
	\label{increasing f}
	Under Assumptions (A\ref{dissipative})-(A\ref{lr}), suppose that there exists an infinite number of intervals $ I_n = [\tau_n, \sigma_n)$. Then, there is a fixed constant $\gamma_2 < \gamma_1$ such that for $n$ large enough,
	\beq
	J(\theta_{\tau_n}) - J(\theta_{\sigma_{n-1}}) \le \gamma_2.
	\eeq
\end{lemma}

\begin{proof}[Proof of Theorem \ref{conv}:]
Recalling \eqref{cycle of time}, we know $\tau_n$ is the first time $|\nabla_\theta J(\theta_t)|> \kappa$ when $t > \sigma_{n-1}$. Thus, for any fixed $\kappa>0$, if there are only a finite number of $\tau_{n}$, then there is a finite $T^{*}$ such that $\left|\nabla_\theta J(\theta_t)\right| \le \kappa$ for $t \ge T^{*}$. We now use a ``proof by contradiction". Suppose that there are infinitely many instances of $\tau_{n}$. By Lemmas \ref{decreasing f} and \ref{increasing f}, we have for sufficiently large $n$ that
	$$
	\begin{aligned}
    	&J\left(\theta_{\sigma_{n}}\right)-J\left(\theta_{\tau_{n}}\right) \le-\gamma_{1} \\
		&J\left(\theta_{\tau_{n}}\right)-J\left(\theta_{\sigma_{n-1}}\right) \le \gamma_{2},
	\end{aligned}
	$$
	where $0<\gamma_{2}<\gamma_{1}$. Choose $N$ large enough so that the above relations hold simultaneously for $n \geq N$. Then,
	\begin{eqnarray}
	J\left(\theta_{\tau_{m+1}}\right) - J\left(\theta_{\tau_{N}}\right) &=& \sum_{n=N}^{m}\left[J\left(\theta_{\sigma_{n}}\right) - J\left(\theta_{\tau_{n}}\right) + J\left(\theta_{\tau_{n+1}}\right) - J\left(\theta_{\sigma_{n}}\right)\right] \notag \\
	&\le& \sum_{k=N}^{n}\left(-\gamma_{1}+\gamma_{2}\right) \notag \\
	&<& (m - N) \times \left(-\gamma_{1}+\gamma_{2}\right).
	\end{eqnarray}
	Letting $m \rightarrow \infty$, we observe that $J\left(\theta_{\tau_{m}}\right) \rightarrow -\infty$, which is a contradiction, since by definition $J(\theta_t) \ge 0$. Thus, there can be at most finitely many $\tau_{n}$. Thus, there exists a finite time $T$ such that almost surely $|\nabla_\theta J(\theta_t)| < \kappa$ for $t \ge T$. Since $\kappa$ is arbitrarily chosen, we have proven that $|\nabla_\theta J(\theta_t)| \to 0$ as $t \to \infty$ almost surely.

\end{proof}

\section*{Acknowledgement}

This research has been supported by the EPSRC Centre for Doctoral Training in Mathematics of Random Systems: Analysis, Modelling and Simulation (EP/S023925/1).

\section*{Appendix}

\appendix
\renewcommand{\appendixname}{Appendix~\Alph{section}}

\section{Proof of Proposition \ref{ergodic estimation}} \label{detailed ergodic}

\begin{itshape}Proof of (\romannumeral1).\end{itshape} 
\eqref{theta decay} with $i=0$ holds from Lemma \ref{ergodic}. And for $i=1$, define 
$$
\hat f(t, x,\theta) = \e f(X_t^{\theta, x})\quad \text{and} \quad \tilde f_{t_0}(t, x, \theta) = \hat f(t,x, \theta) - \hat f(t+t_0, x, \theta)
$$
and we have 
$$
\lim_{t_0 \to \infty} \tilde f_{t_0}(t, x, \theta) = \e f(X_t^{\theta, x}) - \e_{\pi_\theta}f(Y) 
$$
Note that by markov property of ${X^\theta_\cdot}$, we have 
\beq
\tilde f_{t_0}(t,x,\theta) = \hat f(t,x, \theta) - \e f(X_{t+t_0}^{\theta, x}) = \hat f(t,x, \theta) - \e \left[ \e\left[ f(X_{t+t_0}^{\theta, x})\Big| \mathscr{F}_{t_0} \right] \right] = \hat f(t,x, \theta) - \e \left[ \hat f(t, X_{t_0}^{\theta, x}, \theta) \right] 
\eeq
Then we obtain 
\beq
\label{gradient theta}
\nabla_\theta \tilde f_{t_0}(t,x,\theta) = \nabla_\theta \hat f(t,x, \theta) - \e \left[ \nabla_\theta \hat f(t, X_{t_0}^{\theta, x}, \theta) \right] - \e \left[ \nabla_x \hat f(t, X_{t_0}^{\theta, x}, \theta) \nabla_\theta X_{t_0}^{\theta, x} \right].
\eeq
We need the following statement
\begin{itemize}
    \item For any $t \ge 0, x \in \mathbb{R}^{d}, \theta \in \mathbb{R}^{\ell}$
    \beq
    \label{partial x}
    \left|\nabla_{x} \hat{f}(t, x, \theta)\right| \le C e^{-\frac{\beta t}{2}} .
    \eeq
    \item There exist $\eta>0$ such that for any $t \ge 0, \theta \in \mathbb{R}^{\ell}, x_1, x_2 \in \mathbb{R}^{d}$
    \beq
    \label{partial theta}
    \left|\nabla_\theta \hat{f}\left(t, x_1, \theta \right) - \nabla_\theta \hat{f}\left(t, x_2, \theta \right)\right| \le C e^{-\frac\beta4 t}\left|x_{1}-x_{2}\right| .
    \eeq
\end{itemize}

For the first statement, we have by Lemma \ref{difference lemma}
$$
\left|\hat{f} \left(t, x_1, \theta \right)- \hat{f}\left(t, x_2, \theta \right)\right| =\left|\e f( X_{t}^{\theta, x_1} )- \e f(X_{t}^{\theta, x_2})\right| \le C \e \left| X_{t}^{\theta, x_1} - X_{t}^{\theta, x_2} \right|  \le C e^{-\frac\beta2 t}\left|x_{1}-x_{2}\right|,
$$
which implies \eqref{partial x}. Then for the second statement, assumptions A\ref{dissipative} and A\ref{first derivative} imply that $X_t^{\theta, x}$ is differentiable w.r.t $\theta$ and its derivative $\nabla_{\theta} X_t^{\theta, x}$ satisfies 

\begin{equation}
    \label{gradient theta sde}
	\left\{
	\begin{aligned}
		d \nabla_{\theta} X_t^{\theta, x} &= \left[ \nabla_{x} \mu(X_t^{\theta, x}, \theta) \nabla_\theta X_t^{\theta, x} + \nabla_\theta \mu(X_t^{\theta, x}, \theta)\right] dt + \left[ \nabla_{x} \sigma(X_t^{\theta, x}, \theta) \nabla_\theta X_t^{\theta, x} + \nabla_\theta \sigma(X_t^{\theta, x}, \theta) \right]dW_t, \\
        \nabla_{\theta} X_0^{\theta, x}&= 0.
	\end{aligned}
	\right.
\end{equation}
where the SDE in \eqref{gradient theta sde} can be written explicitly as \bae
d \frac{\partial X_t^{\theta,x,m}}{\partial \theta_n} &= \left[ \nabla_x \mu_m(X_t^{\theta, x}, \theta) \frac{\partial X_t^{\theta,x}}{\partial \theta_n} + \frac{\partial \mu_n(X_t^{\theta,x}, \theta)}{\partial \theta_n} \right] dt + \sum_{k=1}^d \left[ \nabla_x \sigma_{mk}(X_t^{\theta, x}, \theta) \frac{\partial X_t^{\theta,x}}{\partial \theta_n} + \frac{\partial \sigma_{mk}(X_t^{\theta,x}, \theta)}{\partial \theta_n} \right] dW_t^k
\eae
for $m \in \{1,2,\cdots, d\}$ and $n \in \{1,2,\cdots, \ell\}$. Due to the dissipativity of $\mu$, for any $x, y$ and $t>0$, we have that 
\beq
\left\langle \mu(x+ty, \theta) - \mu(x, \theta),\ ty \right\rangle + \frac32 \left| \sigma(x+ty, \theta) - \sigma(x, \theta) \right|^2 \le -\beta t^2|y|^2.
\eeq
Therefore,
\bae
&\frac{1}{t^2} \left\langle \int_0^t \frac{d}{ds} \mu(x+sy, \theta) ds,\ ty \right\rangle + \frac32 \left| \frac1t \int_0^t \frac{d}{ds} \sigma(x+sy, \theta) ds \right|^2 \\
=& \left\langle \frac1t \int_0^t \mu_x(x+sy, \theta)y ds,\ y \right\rangle + \frac32 \left|\frac1t \int_0^t \sigma_x(x+sy, \theta)y ds \right|^2 \\
\le& -\beta|y|^2.
\eae
Taking the limit $t\to 0^+$ yields
\beq
\left\langle \mu_x(x, \theta)y, \, y\right\rangle + \frac32 |\sigma_x(x, \theta) y |^2 \le -\beta |y|^2,\quad \forall x, y \in \mathbb{R}^d.
\label{MuDerivInequality}
\eeq

Combining \eqref{gradient theta sde} and \eqref{MuDerivInequality}, for any $k \in \{1,2,\cdots, \ell\}$
\bae
\label{grad theta cal}
&\frac{d}{dt} \e \left| \frac{\partial X_t^{\theta, x}}{\partial \theta_k} \right|^2 \\
=& \e \left[ 2\left\langle \nabla_{x} \mu(X_t^{\theta, x}, \theta) \frac{\partial X_t^{\theta, x}}{\partial \theta_k} + \frac{\partial \mu(X_t^{\theta, x}, \theta)}{\partial \theta_k},\ \frac{\partial X_t^{\theta, x}}{\partial \theta_k} \right\rangle + \left| \nabla_{x} \sigma(X_t^{\theta, x}, \theta) \frac{\partial X_t^{\theta, x}}{\partial \theta_k} + \frac{\partial  \sigma(X_t^{\theta, x}, \theta)}{\partial \theta_k} \right|^2 \right] \\
=& \e \left[ 2\left\langle \nabla_{x} \mu(X_t^{\theta, x}, \theta) \frac{\partial X_t^{\theta, x}}{\partial \theta_k},\ \frac{\partial X_t^{\theta, x}}{\partial \theta_k} \right\rangle + 2\left| \nabla_{x} \sigma(X_t^{\theta, x}, \theta) \frac{\partial X_t^{\theta, x}}{\partial \theta_k} \right|^2 +  2\left\langle \frac{\partial\mu(X_t^{\theta, x}, \theta)}{\partial \theta_k},\ \frac{\partial X_t^{\theta, x}}{\partial \theta_k} \right\rangle + 2\left| \frac{\partial \sigma(X_t^{\theta, x}, \theta)}{\partial \theta_k} \right|^2 \right]\\
\le& -2\beta \e \left| \frac{\partial X_t^{\theta, x}}{\partial \theta_k} \right|^2 + \beta \e\left| \frac{\partial X_t^{\theta, x}}{\partial \theta_k} \right|^2 + \left( \frac1\beta \e \left| \frac{\partial \mu(X_t^{\theta, x}, \theta)}{\partial \theta_k}\right|^2 + 2\e\left| \frac{\partial \sigma(X_t^{\theta, x}, \theta)}{\partial \theta_k} \right|^2  \right)\\
\overset{(a)}{\le}& -\beta \e \left| \frac{\partial X_t^{\theta, x}}{\partial \theta_k} \right|^2 + C,
\eae
where step $(a)$ is by the uniform boundedness of $\nabla_\theta \mu, \nabla_\theta \sigma$ from \eqref{Lip}. Thus 
\beq
\label{grad theta bound}
\sup_{\theta \in \mathbb{R}^\ell, x \in \mathbb{R}^d} \left| \nabla_\theta X_t^{\theta, x} \right|^2 \le C, \quad \forall t \ge 0.
\eeq
Then
\bae
&\left|\nabla_\theta \hat{f}\left(t, x_1, \theta \right) - \nabla_\theta \hat{f}\left(t, x_2, \theta \right)\right|\\
=& \left|\nabla_\theta \e f(X_t^{\theta, x_1}) - \nabla_\theta \e f(X_t^{\theta, x_2}) \right| \\
=& \left| \e \left[ \nabla f(X_t^{\theta, x_1}) \nabla_\theta X_t^{\theta, x_1} \right] - \e \left[ \nabla f(X_t^{\theta, x_2}) \nabla_\theta X_t^{\theta, x_2} \right] \right| \\
\le& \left| \e  \left[ \nabla f(X_t^{\theta, x_1}) \nabla_\theta X_t^{\theta, x_1} \right] - \e \left[ \nabla f(X_t^{\theta, x_2}) \nabla_\theta X_t^{\theta, x_1}\right] \right| + \left| \e \left[ \nabla f(X_t^{\theta, x_2}) \nabla_\theta X_t^{\theta, x_1}\right] - \e \left[ \nabla f(X_t^{\theta, x_2}) \nabla_\theta X_t^{\theta, x_2} \right] \right|\\
:=& I_1 + I_2
\eae
For the terms $I_{1}$, it follows from Assumption A\ref{f}, Lemma \ref{difference lemma} and \eqref{grad theta bound} that 
\bae
\label{diff 1}
I_{1} \le C \left[ \e\left| X_t^{\theta, x_1} - X_t^{\theta, x_2}\right|^{2}\right]^{\frac12} \cdot \left[ \e \left| \nabla_\theta X_t^{\theta, x_1}\right|\right] ^{\frac12} \le C e^{-\frac\beta2 t} \left|x_{1}-x_{2}\right|.
\eae
Then for $I_2$, by the same calculation in Lemma \ref{difference lemma} we have 
\beq
\label{grad theta lip}
\e \left| \nabla_\theta X_t^{\theta, x_1} - \nabla_\theta X_t^{\theta, x_2} \right|^2 \le e^{-\frac\beta2 t} \left| x_1 - x_2 \right|^{2}.
\eeq
Actually, define $Y_t = \nabla_\theta X_t^{\theta, x_1} - \nabla_\theta X_t^{\theta, x_2}$ and Let $Y_t^{:,k}$ denote its k-th column, i.e. 
$$
Y_t^{:,k} := \frac{\partial X_t^{\theta,x_1}}{\partial \theta_k} - \frac{\partial X_t^{\theta,x_2}}{\partial \theta_k}
$$
Then from \eqref{gradient theta sde} we have for any $k \in \{1,2,\cdots, \ell\}$
\bae
\label{grad theta lip cal}
&\frac{d}{dt} \e \left| Y^{:,k}_t \right|^2 \\
=& \e \left[ 2\left\langle \nabla_{x} \mu(X_t^{\theta, x_1}, \theta) \frac{\partial X_t^{\theta, x_1}}{\partial \theta_k} + \frac{\partial \mu(X_t^{\theta, x_1}, \theta)}{\partial \theta_k} - \nabla_{x} \mu(X_t^{\theta, x_2}, \theta) \frac{\partial X_t^{\theta, x_2}}{\partial \theta_k} - \frac{\partial \mu(X_t^{\theta, x_2}, \theta)}{\partial \theta_k},\ Y^{:,k}_t \right\rangle \right]\\
+& \e \left[ \left| \nabla_{x} \sigma(X_t^{\theta, x_1}, \theta) \frac{\partial X_t^{\theta, x_1}}{\partial \theta_k} + \frac{\partial \sigma(X_t^{\theta, x_1}, \theta)}{\partial \theta_k} - \nabla_{x} \sigma(X_t^{\theta, x_2}, \theta) \frac{\partial X_t^{\theta, x_2}}{\partial \theta_k} - \frac{\partial \sigma(X_t^{\theta, x_2}, \theta)}{\partial \theta_k} \right|^2 \right] \\
\le& \e \left[ 2\left\langle \nabla_{x} \mu(X_t^{\theta, x_1}, \theta) Y^{:,k}_t,\ Y^{:,k}_t \right\rangle + 2\left| \nabla_{x} \sigma(X_t^{\theta, x_1}, \theta) Y^{:,k}_t \right|^2\right]
+ \beta \e |Y^{:,k}_t|^2 + C \e \left| X_t^{\theta, x_1} - X_t^{\theta, x_2}\right|^2 \\
\le& -\beta \e|Y^{:,k}_t|^2 + Ce^{-\beta t}|x_1 - x_2|^2,
\eae
which derives \eqref{grad theta lip}. Noting that from Assumption A\ref{f} we have $|\nabla f| \le C$ and thus 
\beq
\label{diff 2}
I_2 \le C \e\left| \nabla_\theta X_t^{\theta, x_1} - \nabla_\theta X_t^{\theta, x_2} \right| \le C e^{-\frac{\beta}{4}t} \left| x_1 - x_2 \right|.
\eeq
Therefore, \eqref{diff 1} and \eqref{diff 2} imply \eqref{partial theta}. 

Finally, by estimate \eqref{gradient theta}, \eqref{partial x}, \eqref{partial theta} and \eqref{grad theta bound}, we have
\beq
| \nabla_\theta \tilde f_{t_0}(t,x,\theta) | \le C e^{-\frac\beta4 t} \e | x-X_{t_0}^{\theta, x}| + Ce^{-\frac\beta2 t} \le C e^{-\frac\beta4 t} (1 + |x|), \quad \forall t_0 \ge 0.
\eeq
Let $t_0 \to \infty$ and by DCT, we have 
\beq
\left| \nabla_\theta \e f(X_t^{\theta, x}) - \nabla_\theta \e_{\pi_\theta}f(Y) \right| \le C e^{-\frac\beta4 t} (1 + |x|).
\eeq

The proof of \eqref{theta decay} for $i=2$ follows from the same idea. Actually, from \eqref{gradient theta} we know 
\bae
\label{gradient theta 2}
\left| \nabla^2_\theta \tilde f_{t_0}(t,x,\theta)\right| \le&\left| \nabla^2_\theta \hat f(t,x, \theta) - \e \left[ \nabla^2_\theta \hat f(t, X_{t_0}^{\theta, x}, \theta) \right] \right| \\
+& \left| \e \left[ \nabla_x \nabla_\theta \hat f(t, X_{t_0}^{\theta, x}, \theta) \nabla_\theta X^{\theta, x}_{t_0} \right] \right| \\
+& \left| \e \left\langle \nabla_\theta \nabla_x \hat f(t, X_{t_0}^{\theta, x}, \theta) + \nabla^2_x \hat f(t, X_{t_0}^{\theta, x}, \theta) \nabla_\theta X_{t_0}^{\theta, x}, \ \nabla_\theta X_{t_0}^{\theta, x} \right\rangle \right|\\
+& \left| \e \left[ \nabla_x \hat f(t, X_{t_0}^{\theta, x}, \theta) \nabla^2_\theta X_{t_0}^{\theta, x} \right] \right|\\
:=& I_3 + I_4 + I_5 + I_6.
\eae

First for the term $I_3$, note that
$$
\nabla_\theta \hat f(t, x, \theta) = \nabla_\theta \e f(X_t^{\theta, x}) = \e \left[ \nabla f(X_t^{\theta, x}) \nabla_\theta X_t^{\theta, x} \right],
$$
and thus
$$
\begin{aligned}
\nabla_\theta^2 \hat f(t,x,\theta) = \e \left\langle \nabla^2 f(X_t^{\theta, x})  \nabla_\theta X_t^{\theta, x},\  \nabla_\theta X_t^{\theta, x}  \right\rangle + \e \left[ \nabla f(X_t^{\theta,x}) \nabla_\theta^2 X_t^{\theta, x} \right],
\end{aligned}
$$
where $\nabla f(X_t^\theta) \nabla_\theta^2 f(X_t^\theta)$ is defined as:
\beq
\label{element wise}
\left[ \nabla f(X_t^\theta) \nabla_\theta^2 f(X_t^\theta) \right]_{m,n} = \nabla f(X_t^\theta) \frac{\partial^2 X_t^{\theta, x}}{\partial \theta_m \partial \theta_n}.
\eeq
Then for any $x_{1}, x_{2} \in \mathbb{R}^{d}$,
\bae
\label{decomposition}
&\left| \nabla_\theta^2 \hat f(t,x_1,\theta) - \nabla_\theta^2 \hat f(t,x_2,\theta) \right|\\
\le&  \left| \e \left\langle \nabla^2 f(X_t^{\theta, x_1}) \nabla_\theta X_t^{\theta, x_1},\ \nabla_\theta X_t^{\theta, x_1} \right\rangle 
- \left\langle \nabla^2 f(X_t^{\theta, x_2})  \nabla_\theta X_t^{\theta, x_2},\ \nabla_\theta X_t^{\theta, x_2} \right\rangle \right| \\
+& \left| \e \left[ \nabla f(X_t^{\theta,x_1}) \nabla_\theta^2 X_t^{\theta, x_1}
- \e \nabla f(X_t^{\theta,x_2}) \nabla_\theta^2 X_t^{\theta, x_2} \right] \right| \\
:=& I_{3,1} + I_{3,2}.
\eae

As in \eqref{grad theta cal}, by It\^{o}'s formula for any $k \in \{1,2,\cdots, \ell\}$
\bae
\label{grad theta 4}
&\frac{d}{dt} \e \left| \frac{\partial X_t^{\theta, x}}{\partial \theta_k} \right|^4 \\
\le& \e \left[ \left|\frac{\partial X_t^{\theta, x}}{\partial \theta_k}\right|^2 \left( 4\left\langle \nabla_x \mu(X^{\theta, x}_t, \theta)\frac{\partial X_t^{\theta, x}}{\partial \theta_i} + \frac{\partial \mu(X^{\theta, x}_t, \theta)}{\partial \theta_k}, \ \frac{\partial X_t^{\theta, x}}{\partial \theta_k} \right\rangle + 6\left|\nabla_x \sigma(X^{\theta, x}_t, \theta) \frac{\partial X_t^{\theta, x}}{\partial \theta_k} + \frac{\partial \sigma(X^{\theta, x}_t, \theta)}{\partial \theta_k}\right|^2 \right) \right] \\
\overset{(a)}{\le}& \e \left[ \left|\frac{\partial X_t^{\theta, x}}{\partial \theta_k}\right|^2 \left( -4\beta \left| \frac{\partial X_t^{\theta, x}}{\partial \theta_k} \right|^2 + 2\beta \left|\frac{\partial X_t^{\theta, x}}{\partial \theta_k}\right|^2 + C\left( \left| \frac{\partial \mu(X_t^{\theta, x}, \theta)}{\partial \theta_k} \right|^2 + \left| \frac{\partial \sigma(X_t^{\theta, x}, \theta)}{\partial \theta_k} \right|^2  \right) \right) \right] \\
\le& -2\beta \e \left| \frac{\partial X_t^{\theta, x}}{\partial \theta_k} \right|^4 + C \e \left| \frac{\partial X_t^{\theta, x}}{\partial \theta_k} \right|^2,
\eae
where step $(a)$ use the same calculations as in \eqref{trick}. Thus combining \eqref{grad theta bound} and \eqref{grad theta 4}, we have 
\beq
\label{grad theta bound 4}
\sup_{\theta \in \mathbb{R}^\ell, x \in \mathbb{R}^d} \left| \nabla_\theta X_t^{\theta, x} \right|^4 \le C, \quad \forall t \ge 0.
\eeq
Then for $I_{3,1}$, by direct computation, we know there exists $\gamma>0$ such that 
\bae
\label{I31}
I_{3,1} &\le \left| \e \left[ \left\langle \nabla^2 f(X_t^{\theta, x_1}) \nabla_\theta X_t^{\theta, x_1}, \ \nabla_\theta X_t^{\theta, x_1} \right\rangle 
- \left\langle \nabla^2 f(X_t^{\theta, x_2}) \nabla_\theta X_t^{\theta, x_1}, \ \nabla_\theta X_t^{\theta, x_1}\right\rangle \right] \right|\\
&+ \left| \e \left[ \left\langle \nabla^2 f(X_t^{\theta, x_2}) \nabla_\theta X_t^{\theta, x_1},\ \nabla_\theta X_t^{\theta, x_1} \right\rangle 
- \left\langle \nabla^2 f(X_t^{\theta, x_2})\nabla_\theta X_t^{\theta, x_2},\ \nabla_\theta X_t^{\theta, x_2} \right\rangle \right] \right| \\
&\le \e \left| \nabla^2 f(X_t^{\theta, x_1}) - \nabla^2 f(X_t^{\theta, x_2}) \right| \cdot \left| \nabla_\theta X_t^{\theta, x_1} \right|^2 + \e \left| \nabla^2 f(X_t^{\theta, x_2}) \right|\cdot \left| \nabla_\theta X_t^{\theta, x_1} - \nabla_\theta X_t^{\theta, x_2} \right| \cdot \left( \left| \nabla_\theta X_t^{\theta, x_1} \right| + \left| \nabla_\theta X_t^{\theta, x_2} \right| \right)\\
&\le C \left( \e \left| X_t^{\theta, x_1} -X_t^{\theta, x_2} \right|^2\right)^{\frac12} \cdot \left( \e \left| \nabla_\theta X_t^{\theta, x_1} \right|^4\right)^{\frac12} + C \left( \e \left| \nabla_\theta X_t^{\theta, x_1} - \nabla_\theta X_t^{\theta, x_1} \right|^2\right)^{\frac12} \cdot \left( \e\left| \nabla_\theta X_t^{\theta, x_1} \right|^2 + \e \left| \nabla_\theta X_t^{\theta, x_2} \right|^2 \right)^{\frac12}\\
&\overset{(a)}{\le} C e^{-\gamma t} |x_1-x_2|,
\eae
where step $(a)$ we use Lemma \ref{difference lemma}, \eqref{grad theta lip} and \eqref{grad theta bound 4}.

Under the assumptions A\ref{dissipative}-A\ref{high derivative} and using the same calculations as in \eqref{grad theta cal} and \eqref{grad theta lip cal}, it is easy to prove that
\beq
\label{grad theta 2 bound}
\sup_{x\in \mathbb{R}^d, \theta \in \mathbb{R}^\ell} \e \left| \nabla^2_\theta X_t^{\theta, x} \right|^2 \le C, \quad \forall t \ge 0,
\eeq
and there exists $\gamma>0$ such that 
\beq 
\label{grad theta 2 lip}
\sup_{ \theta \in \mathbb{R}^\ell} \e \left| \nabla^2_\theta X_t^{\theta, x_1} - \nabla^2_\theta X_t^{\theta, x_2} \right|^2 \le Ce^{-\gamma t} |x_1-x_2|,\quad \forall x_1, x_2 \in \mathbb{R}^d,\ t \ge0.
\eeq
Thus we have 
\beq
\label{I32}
I_{3,2} \le 
\left| \e \nabla f(X_t^{\theta,x_1}) \left( \nabla_\theta^2 X_t^{\theta, x_1}
- \nabla_\theta^2 X_t^{\theta, x_2} \right) \right| + \left| \e \left( \nabla f(X_t^{\theta,x_1}) - \nabla f(X_t^{\theta,x_2}) \right) \nabla_\theta^2 X_t^{\theta, x_2}  \right| \le Ce^{-\gamma t} |x_1-x_2|.
\eeq
Combining \eqref{decomposition}, \eqref{I31} and \eqref{I32}, we have 
\beq
\label{I3}
I_{3} \le C e^{-\gamma t}\left( \e \left|x - X_{t_0}^{\theta, x}\right|\right) \le C e^{-\gamma t} \left(1+|x|\right) .
\eeq

For the term $I_{4}$, note that
$$
\nabla_x \nabla_\theta \hat f(t ,x ,\theta) = \e \left\langle \nabla^2 f(X_t^{\theta, x}) \nabla_x X_t^{\theta, x},\ \nabla_\theta X_t^{\theta, x} \right\rangle + \e \nabla f(X_t^{\theta, x}) \nabla_x \nabla_\theta X_t^{\theta, x}.
$$
By Lemma \ref{difference lemma} and \eqref{grad theta lip}, we have 
$$
\sup_{x \in \mathbb{R}^{d}, \theta \in \mathbb{R}^{m}} \left(\e \left|\nabla_x X_t^{\theta, x}\right|^{2} + \e \left| \nabla_x \nabla_\theta X_t^{\theta, x} \right|^{2}\right) \le C e^{-\frac{\beta}{2}t}, \quad t \ge 0, 
$$
which derives that 
\beq
\left| \nabla_x \nabla_\theta \hat f(t,x,\theta) \right| \le C e^{-\frac\beta 4 t}. 
\eeq
Hence, it is easy to see that 
\beq
\label{I4}
I_{4} \le C e^{-\frac{\beta}{4}t} \e \left| \nabla_\theta X^{\theta, x}_{t_0}\right| \le C e^{-\frac{\beta}{4}t}.
\eeq

For the term $I_5$, by a similar argument as for $I_4$, we have
$$
\sup_{x \in \mathbb{R}^{d}, \theta \in \mathbb{R}^{m}} \e \left|\nabla^2_x X_t^{\theta, x}\right|^{2} \le C e^{-\beta t}, \quad t \ge 0.
$$
and then
\beq
\label{Xdecay}
\left|\nabla^2_{x} \hat{f}(t, x, \theta)\right| + \left|\nabla_{x}\nabla_\theta \hat{f}(t, x, \theta)\right| \le C e^{-\frac{\beta t}{4}} .
\eeq
Hence, 
\beq
\label{I5}
I_{5} \le C e^{-\frac{\beta t}{4}}.
\eeq
Finally for $I_{6}$, from \eqref{partial x} and \eqref{grad theta 2 bound} we can get
\beq
\label{I6}
I_{6} \le C e^{-\frac{\beta t}{2}}.
\eeq
Hence, combining \eqref{I3}, \eqref{I4}, \eqref{I5} and \eqref{I6} there exists $\gamma>0$ such that 
\beq
| \nabla^2_\theta \tilde f_{t_0}(t,x,\theta) | \le Ce^{-\gamma t} \le C e^{-\gamma t} (1 + |x|), \quad \forall t_0 \ge 0.
\eeq
Let $t_0 \to \infty$, we have
\beq
\left| \nabla^2_\theta \e f(X_t^{\theta, x}) - \nabla^2_\theta \e_{\pi_\theta}f(Y) \right| \le C e^{-\gamma t} (1 + |x|).
\eeq

\begin{itshape} Proof of (\romannumeral2) and (\romannumeral3).\end{itshape}
First note that by Lemma \ref{moment stable}
$$
\left| \e f(X_t^{\theta, x}) \right| \le C(1+\e\left|X_t^{\theta, x}\right|) \le C\left(1+e^{-\frac\beta2 t} |x|\right),
$$
which together with Lemma \ref{ergodic} derive
\beq
\left| \e_{\pi_\theta} f(Y) \right| = \lim\limits_{t\to \infty} \left| \e f(X_t^{\theta, x}) \right| \le C.
\eeq
By \eqref{grad theta bound}, we have 
\beq
\left| \nabla_\theta \e f(X_t^{\theta, x}) \right| = \left| \e \nabla f(X_t^{\theta, x}) \nabla_\theta X_t^{\theta, x} \right| \le C,
\eeq
which together with \eqref{theta decay} derive
\beq
\left| \nabla_\theta \e_{\pi_\theta} f(Y) \right|= \left| \lim\limits_{t\to \infty} \nabla_\theta \e f(X_t^{\theta, x}) \right| \le C.
\eeq
Similarly we can get the bound for $\left| \nabla^2_\theta \e f(X_t^{\theta, x}) \right|$ and derive \eqref{invariant density} for $i=2$.

Then for \eqref{theta x decay}, the case for $j =1$ directly follows from \eqref{partial x}, \eqref{partial theta} and the case for $i=0, j=2$ follows from \eqref{Xdecay}. Finally, it is easy to prove there exists $\gamma>0$ such that for $i,j \in\{ 0,1 \}$
\bae
\sup_{ \theta \in \mathbb{R}^\ell} \e \left| \nabla^j_x \nabla_\theta^i X_t^{\theta, x_1} - \nabla^j_x \nabla_\theta^i X_t^{\theta, x_2} \right|^2 \le Ce^{-\gamma t} |x_1-x_2|,\quad \forall x_1, x_2 \in \mathbb{R}^d,\ t \ge0.
\eae
Thus by the same method to estimate \eqref{decomposition}, we can get there exists $\eta>0$ such that 
\bae
&\left| \nabla_x \nabla_\theta \hat f(t,x_1,\theta) - \nabla_x \nabla_\theta \hat f(t,x_2,\theta) \right|\\
\le&  \left| \e \left[ \left\langle \nabla^2 f(X_t^{\theta, x_1}) \nabla_x X_t^{\theta, x_1},\ \nabla_\theta X_t^{\theta, x_1} \right\rangle 
- \left\langle \nabla^2 f(X_t^{\theta, x_2}) \nabla_x X_t^{\theta, x_2},\ \nabla_\theta X_t^{\theta, x_2} \right\rangle \right] \right| \\
+& \left| \e \left[ \nabla f(X_t^{\theta,x_1}) \nabla_x \nabla_\theta X_t^{\theta, x_1} 
- \nabla f(X_t^{\theta,x_2}) \nabla_x \nabla_\theta X_t^{\theta, x_2} \right] \right| \\
\le& C e^{-\gamma t}|x_1 - x_2|,
\eae
which derives the case for $i=1, j=2$ and thus the proof is completed.

\bibliographystyle{plain}

\bibliography{cite}

\end{document}